\documentclass[10pt]{article}
\usepackage{amssymb}
\usepackage{amsfonts}
\usepackage{graphicx}
\usepackage{amsmath}
\usepackage{xr}
\usepackage{psfrag}
\usepackage{color,soul,textcomp}
\usepackage{enumerate}

\newtheorem{theorem}{Theorem}[section]

\newtheorem{assumption}[theorem]{Assumption}

\newtheorem{claim}[theorem]{Claim}

\newtheorem{corollary}[theorem]{Corollary}

\newtheorem{example}[theorem]{Example}

\newtheorem{lemma}[theorem]{Lemma}

\newtheorem{proposition}[theorem]{Proposition}
\newtheorem{remark}[theorem]{Remark}

\newcommand{\R}{\mathbb{R}}
\newenvironment{proof}[1][Proof]{\noindent\textit{#1.} }{\hfill \rule{0.5em}{0.5em}}

\begin{document}

\title{\textbf{Spreading speeds for multidimensional reaction-diffusion systems of the prey-predator type}}
\author{\textsc{Arnaud Ducrot$^a$, Thomas Giletti$^{b}$ and Hiroshi Matano$^c$} \\
$^{a}${\small Normandie Univ, UNIHAVRE, LMAH, FR-CNRS-3335, ISCN, 76600 Le Havre, France}\\
$^b${\small Univ. Lorraine, Institut Elie Cartan Lorraine, UMR 7502, Vandoeuvre-l\`{e}s-Nancy, F-54506, France}\\
$^c${\small Meiji Institute for Advanced Study of Mathematical Sciences}\\
{\small Meiji University, 4-21-1 Nakano, Tokyo 164-8525, Japan}
}

\maketitle

\begin{abstract}
We investigate spreading properties of solutions of a large class of two-component reaction-diffusion systems, including prey-predator systems as a special case. By spreading properties we mean the long time behaviour of solution fronts that start from localized (i.e. compactly supported) initial data. Though there are results in the literature on the existence of travelling waves for such systems, very little has been known ---~at least theoretically~--- about the spreading phenomena exhibited by solutions with compactly supported initial data. The main difficulty comes from the fact that the comparison principle does not hold for such systems. Furthermore, the techniques that are known for travelling waves such as fixed point theorems and phase portrait analysis do not apply to spreading fronts. In this paper, we first prove that spreading occurs with definite spreading speeds. Intriguingly, two separate fronts of different speeds may appear in one solution ---~one for the prey and the other for the predator~--- in some situations.

\vspace{0.2in}\noindent \textbf{Key words} Long time behaviour, spreading speeds, reaction-diffusion, prey-predator systems.

\vspace{0.1in}\noindent \textbf{2010 Mathematical Subject Classification} 35K57, 35B40,
92D30.

\end{abstract}

\section{Introduction}

The aim of this work is to study the asymptotic spreading speed for a class of two-component reaction-diffusion systems including prey-predator models. The systems we consider are given in the form
\begin{equation}\label{1.1}
\begin{cases}
\left(\partial_t -d\Delta\right)u=u F\left(u,v \right)\\
\left(\partial_t-\Delta \right)v=v G\left(u,v \right)
\end{cases}\;\;\text{ for $t>0$ and $x\in \R^N$}.
\end{equation}
This system is supplemented by the initial conditions:
\begin{equation}\label{initial_datum}
u(0,x)\equiv u_0(x),\;\;v(0,x)\equiv v_0(x),
\end{equation}
where $u_0$, $v_0$ are bounded nonnegative functions with compact support. In \eqref{1.1}, $d>0$ is a given positive constant and
we shall assume throughout this paper that the nonlinearities satisfy the following:
\begin{assumption}\label{ASS-F}
The function $F:[0,\infty)\times [0,\infty)\to \R$ is of class $C^1$ and satisfies:
\begin{itemize}
\item [$(a)$] for each $u> 0$, the map $v\mapsto F(u,v)$ is strictly decreasing;
\item [$(b)$] $F(1,0)=0$ and for each $u\in [0,1)$, $F(u,0)>0$ (monostability);
\item [$(c)$] for each $u \geq 0$, $F(u,0) \leq F(0,0)$ (weakly KPP hypothesis);
\end{itemize}
\end{assumption}
and
\begin{assumption}\label{ASS-G}
The function $G:[0,\infty)\times [0,\infty)\to \R$ is of class $C^1$ and satisfies
\begin{itemize}
\item [$(a)$] for each $v\geq 0$, the map $u\mapsto G(u,v)$ is nondecreasing;
\item [$(b)$] $G(0,0) < 0 < G(1,0)$;
\item [$(c)$] for each $u \geq 0$, the map $v \mapsto G(u,v)$ is nonincreasing (strongly KPP hypothesis).
\end{itemize} 
\end{assumption}
The above set of assumptions is mainly motivated by prey-predator models in ecology, but such systems also arise in other fields of science, including models in the combustion theory, chemistry or epidemiology. In the context of prey-predator systems, $u(t,x)$ and $v(t,x)$ typically denote, respectively, the density of the prey and that of the predator.

As we mentioned before, the goal of this paper is to study the large time behaviour, in particular spreading properties, of solutions of \eqref{1.1} under Assumptions~\ref{ASS-F}, \ref{ASS-G}, and the weak dissipativity Assumption~\ref{ASS-U} which we will introduce below.

Spatial propagation for reaction-diffusion systems for which a comparison principle holds has been the subject of an important amount of works. We emphasize that front propagation in scalar reaction-diffusion equations has been widely studied and especially for KPP type nonlinearities. We refer for instance to Kolmogorov, Petrovsky and Piskunov~\cite{KPP}, Fisher~\cite{Fisher}, Aronson and Weinberger \cite{AW78},  
Weinberger \cite{Weinberger-02} (for periodic medium), Berestycki et al \cite{BHNa} for more general medium and the references cited therein. The case of cooperative systems is also quite well understood and we refer to Li et al \cite{LWL}, Lewis et al \cite{Lewis} and the references cited therein. 
We also refer to Liang and Zhao \cite{Liang-Zhao, Liang-Zhao-bis} for general results on monotone semiflows with and without heterogeneities. See also the books \cite{BH} and \cite{VVV} for more exhaustive references on this topic.

As regards prey-predator systems, which are the main subject of the present paper, analysis of the spreading properties becomes much more difficult because of the lack of the comparison principle. Nevertheless, as far as travelling wave solutions are concerned, there is a wide literature on the existence and basic properties of travelling waves of prey-predator systems. We may, for instance, refer the reader to the pioneering works~\cite{Dunbar83,Dunbar86,Gardner}. We also refer to the more recent papers~\cite{Huang,Huang-al}, as well as to~\cite{Li-Xiao} and the references cited therein for a good survey on travelling wave solutions for prey-predator systems. However, travelling waves constitute only a special class of solutions.  While one may expect that the long time behaviour of more general solutions of the Cauchy problem \eqref{1.1}--\eqref{initial_datum} is largely dictated by such travelling waves (at least when they exist, which may depend on the structure of the underlying ODE system), the question of how general solutions of the prey-predator system \eqref{1.1}--\eqref{initial_datum} actually behave has been for the most part an open problem apart from some results on the local stability of travelling waves; see, e.g., \cite{Gardner-Jones}.

How solutions of \eqref{1.1}--\eqref{initial_datum}  behave when starting from localized initial data is an important question in mathematical ecology, epidemiology and other fields of sciences.  Despite its importance, little has been known about the long time behaviour and spreading properties of such solutions, largely because of the lack of the comparison principle when \eqref{1.1} is of the prey-predator type.

There are some -- though not many -- works on the long time behaviour and spreading speeds for other problems for which the comparison principle does not hold. We refer to Weinberger et al \cite{WKS} where a partially cooperative system is studied using some ideas developed by Thieme in \cite{Thieme-79} for scalar integral equations. We also refer to Wang and Castillo-Chavez \cite{Wang} for integro-difference systems and to Fang and Zhao \cite{Fang-Zhao} and the references therein for scalar integro-differential equations. However, these problems are far from prey-predator or epidemic models, and the techniques in those works do not apply to the class of systems we consider. There are some results in \cite{Ducrot, Ducrot-preprint} on the long time behaviour of a special class of prey-predator and epidemic systems but their methods work only for problems of a particular structure and do not apply to our systems. As far as the authors know, the results presented here are the first theoretical results on spreading phenomena for systems of the form \eqref{1.1}-\eqref{initial_datum} under Assumptions~\ref{ASS-F} and~\ref{ASS-G}, or any systems of a similar kind.\\

Before presenting concrete examples of prey-predator systems satisfying the above assumptions, we discuss our hypothesis within this ecological context as follows: 
\begin{itemize}
\item Hypotheses~\textit{$(a)$} of both Assumptions~\ref{ASS-F} and~\ref{ASS-G} correspond to the predation effect. Roughly speaking, the presence of more predators decreases the population of the prey while the presence of more preys increases the population of the predator. Due to this asymmetry, the comparison principle does not hold for the system~\eqref{1.1}.

\item When there is no predator, hypothesis~\textit{$(b)$} of Assumption~\ref{ASS-F} ensures that the prey population grows to the positive stable state $u \equiv 1$ (see also Theorem~\ref{LE1}), which is the carrying capacity of the environment for the prey. 
The condition $G(1,0)>0$ in Assumption~\ref{ASS-G} \textit{$(b)$} means that such an amount of prey is enough to sustain a positive density of predators, while the condition $G(0,0)< 0$ (hence $G(0,v)<  0$ for $v\geq 0$) implies that the predator cannot survive without the prey.


\item Hypotheses~\textit{$(c)$} of Assumptions~\ref{ASS-F} and~\ref{ASS-G} roughly mean that 
the growth rate of the prey and that of the predator reach their maximal value at small densities. This fact suggests that the propagation speed of the two species be determined ``linearly" at the leading edge, as in the scalar KPP equation. As we will see later, this is indeed the case, and it will play an essential role in estimating the spreading speeds of the two species precisely.
If we relax the KPP hypotheses, we expect that the dynamics of the solutions remain partly similar but may become partly more complicated (see Remark \ref{rmk:KPP}).
\end{itemize}

In many typical prey-predator models, the system may be rewritten as:
\begin{equation}\label{predator-prey}
\begin{split}
&\left(\partial_t-d\Delta\right)u=u h\left(u\right)-\Pi\left(u\right)v,\\
&\left(\partial_t-\Delta\right)v=v\left[\mu\Pi\left(u\right)-a\right].
\end{split}
\end{equation}
Function $h:[0,\infty)\to \R$ represents the intrinsic growth rate of the prey while function $\Pi:[0,\infty)\to [0,\infty)$ stands for the functional response to predation or capture rate. It is assumed to depend solely upon the prey density.
Finally $\mu>0$ denotes the conversion rate of biomass, $a>0$ the death rate of predator while $d>0$ denotes the normalized diffusion rate for the prey. Typical examples of functions $h$ and $\Pi$ can be found for instance in \cite{May, OM} and the survey paper of Cheng et al \cite{CHL}. These functions are typically given in the form
\begin{equation}\label{example}
\begin{split}
&h(u):\;\;r\left(1-u \right)\; \mbox{ or }\;\;r\frac{1-u}{1+\varepsilon u},\\
&\Pi(u):\;\;\frac{m u^n}{b+u^n} \mbox{ with } n\geq 1 , \;  \mbox{ or }\;\;m(1-e^{-u}),
\end{split}
\end{equation} 
where $r$, $\varepsilon$, $m$ and $b$ are positive constants.
Another example is
\begin{equation}\label{Lotka}
\begin{split}
&\left(\partial_t-d\Delta\right)u=u \left[1-u-bv\right],\\
&\left(\partial_t-\Delta\right)v=v\left[\mu bu-a\right],
\end{split}
\end{equation}
which is obtained by setting $h(u)=1-u$ and $\Pi(u)=bu$ in \eqref{predator-prey}. Here $a$, $b$, $\mu$ are positive constants.
These examples clearly satisfy Assumptions~\ref{ASS-F} and~\ref{ASS-G} \textit{$(a)$}, \textit{$(c)$}, as well as $G(0,0)<0$. Moreover, the condition $G(1,0)>0$ is satisfied when $\mu$ is sufficiently large. As we will prove in Section~\ref{sec:dissip}, these examples also satisfy the weak dissipativity property below, a concept that will play an important role in our arguments.\\

In order to state our main results, we need to impose one more assumption, namely the weak dissipativity of the semiflow generated by system \eqref{1.1}. To define this concept, we introduce some notation. Let $X={\rm BUC}\left(\R^N,\R^2\right)$ denote the Banach space of $\R^2-$valued bounded and uniformly continuous functions on $\R^N$ endowed with the usual sup-norm. We then define the set $C\subset X$ by
\begin{equation*}
C=\left\{\left(\varphi,\psi\right)\in X:\;0\leq \varphi\leq 1\text{ and }\psi\geq 0\right\}.
\end{equation*}
Our initial data will always be chosen in this set. 

Here, we point out that under Assumptions~\ref{ASS-F} and~\ref{ASS-G}, the set $C$ is positively invariant under \eqref{1.1}, therefore this system generates a strongly continuous nonlinear semiflow $\left\{U(t):C\to C\right\}_{t\geq 0}$. Indeed, although the comparison principle does not hold for the full system~\eqref{1.1}, one can still apply partial comparison arguments on each equation separately. In particular, since 0 is a solution of the $v$-equation for any $u$, it immediately follows that $v(t,x) \geq 0$ for any $t >0$ and $x \in \R^N$, provided that $(u_0,v_0) \in C$. Similarly, one can easily check that $0 \leq  u (t,x) \leq 1$ for any $t>0$ and $x \in \R^N$. 
We also note that, by the strong maximum principle, provided that the functions $u_0$ and $v_0$ are not trivial, we have $0< u(t,x)<1$ and $0< v(t,x)$ for all $t >0$ and $x \in \R^N$.

Now we are ready to state our third assumption:
\begin{assumption}\label{ASS-U}
The nonlinear semiflow $U$ is {\bf weakly dissipative} in $C$, in the sense that for each $\kappa>0$ there exists $M(\kappa)>0$ such that 
\begin{equation*}
U(t)\left[ C\cap B_X (0,\kappa )\right]\subset C\cap B_X (0,M(\kappa)),\;\;\forall t\geq 0,
\end{equation*}
where $B_X (0, \kappa)$ denotes the ball of center $0$ and radius $\kappa$ in the Banach space~$X$.
\end{assumption}

In other words, for any bounded set of initial data, the set of associated solutions of the Cauchy problem remains bounded as $t\to\infty$. The difference of this concept from the usual notion of point dissipativity is that the bound $M(\kappa)$ may depend on $\kappa$.
%
In Section~\ref{sec:dissip}, we will provide general conditions that guarantee the weak dissipativity of the semiflow $\{U(t)\}_{t\geq 0}$. In particular, we will prove that such a property holds true when  $F(\cdot,+\infty)<0$. This condition has a very natural meaning in the ecological context, namely that preys do not survive when there are too many predators.

\section{Main results}\label{sec:main}

In this section, we state the main results that are discussed and proved in this work and we propose an outline of this paper.

\subsection{Uniform spreading}

In this subsection, we present our main results on the spreading properties of solutions of~\eqref{1.1}. We will see that the profile of the solutions differs drastically depending on the values of $c^*$ and $c^{**}$ which we define by:
\begin{equation}\label{speed_prey}
c^* = 2\sqrt{dF(0,0)},
\end{equation}
\begin{equation}\label{speed_predator}
c^{**}=2\sqrt{G(1,0)}.
\end{equation}
The value $c^*$ denotes the spreading speed of the prey $u$ in the absence of the predator. This can be understood by noting that, when $v \equiv 0$, the $u$-equation in \eqref{1.1} becomes a scalar reaction-diffusion equation of the KPP type:
\begin{equation}\label{eq-scalar}
\left(\partial_t -d\Delta\right)u(t,x)=u(t,x)F\left(u(t,x),0\right),\;t>0,\;x\in\R^N.
\end{equation}
This classical case has been studied extensively since the pioneering works~\cite{Fisher,KPP}. In particular, it is well known that for any nonnegative and nontrivial compactly supported initial data, the associated solution spreads with the speed $c^*$ defined above; see \cite{AW78}. We will recall this result more precisely in Theorem~\ref{LE1} in Section~\ref{sec:outer}.

On the other hand, the value $c^{**}$ denotes the spreading speed of the predator $v$ when the prey is abundant, namely when $u \equiv 1$. In other words, $c^{**}$ corresponds to the spreading speed of solutions of 
\begin{equation}\label{eq-scalar-bis}
(\partial_t - \Delta ) v  = v G \left(1,v \right), \; t >0 , \; x \in \R^N, 
\end{equation}
with nonnegative and nontrivial compactly supported initial data. Note that \eqref{eq-scalar-bis} is not a KPP type equation in the usual sense, because we do not assume that it admits a positive stationary state. However, it is completely straightforward to check by using the same arguments as in~\cite{AW78} that solutions exhibit a similar spreading behaviour, the only difference being that they may no longer converge to a stationary state after the propagation but instead grow indefinitely.\\

Under the above notations and assumptions, our first main result deals with the case when the prey spreads faster than the predator:
\begin{theorem}[Slow predator]\label{THEO1}
Let Assumptions \ref{ASS-F}, \ref{ASS-G} and \ref{ASS-U} be satisfied, and $u_0$, $v_0$ be two given nontrivial compactly supported functions such that $(u_0,v_0)\in C$.

If $c^{**}< c^*$, then the solution $(u,v)\equiv \left(u(t,x),v(t,x)\right)$ of \eqref{1.1} with initial data $\left(u_0,v_0\right)$ satisfies:
\begin{itemize}
\item [$(i)$] $\displaystyle\lim_{t\to\infty}\sup_{\|x\|\geq ct} u(t,x)=0$ for all $c>c^*$;

\item [$(ii)$] for all $c^{**}<c_1<c_2<c^*$ and each $c>c^{**}$ one has:
\begin{equation*}
\displaystyle\lim_{t\to\infty}\sup_{c_1t\leq \|x\|\leq c_2t} \left|1-u(t,x)\right|+\sup_{\|x\|\geq ct} v(t,x)=0;
\end{equation*}
\item [$(iii)$] there exists $\varepsilon>0$  such that for each $c\in \left[0,c^{**}\right)$ one has
\begin{equation*}
\liminf_{t\to \infty} \inf_{\| x \| \leq c t } v(t,x) \geq \varepsilon,
\end{equation*}
\begin{equation*}
\limsup_{t\to\infty} \sup_{\| x \| \leq c t }  u(t,x)\leq 1-\varepsilon \text{ and } \liminf_{t\to\infty} \inf_{\| x \| \leq c t }  u(t,x) \geq \varepsilon.
\end{equation*}
\end{itemize}
\end{theorem}

Our second result describes the behaviour of the solutions of \eqref{1.1} when the prey is not able to outrace the predator.
Our precise result reads as:
\begin{theorem}[Fast predator]\label{THEO2}
Let Assumptions \ref{ASS-F}, \ref{ASS-G} and \ref{ASS-U} be satisfied, and $u_0$, $v_0$ be two given nontrivial compactly supported functions such that $(u_0,v_0)\in C$.

If $c^{**}\geq c^* $, then the solution $(u,v)\equiv \left(u(t,x),v(t,x)\right)$ of \eqref{1.1} with initial data $\left(u_0,v_0\right)$ satisfies
\begin{itemize}
\item [$(i)$] $\displaystyle\lim_{t\to\infty}\sup_{\|x\|\geq ct} u(t,x) + v(t,x) =0$ for all $c>c^{*}$;

\item [$(ii)$] there exists $\varepsilon>0$ such that for each $c\in \left[0,c^{*}\right)$ one has
\begin{equation*}
\liminf_{t\to \infty} \inf_{\| x \| \leq c t } v(t,x) \geq \varepsilon,
\end{equation*}
\begin{equation*}
\limsup_{t\to\infty} \sup_{\| x \| \leq c t }  u(t,x)\leq 1-\varepsilon \text{ and } \liminf_{t\to\infty} \inf_{\| x \| \leq c t }  u(t,x) \geq \varepsilon.
\end{equation*}\end{itemize}
\end{theorem}

\begin{tabular}{c}
\includegraphics[bb=0 0 1079 306, width=0.96\textwidth]{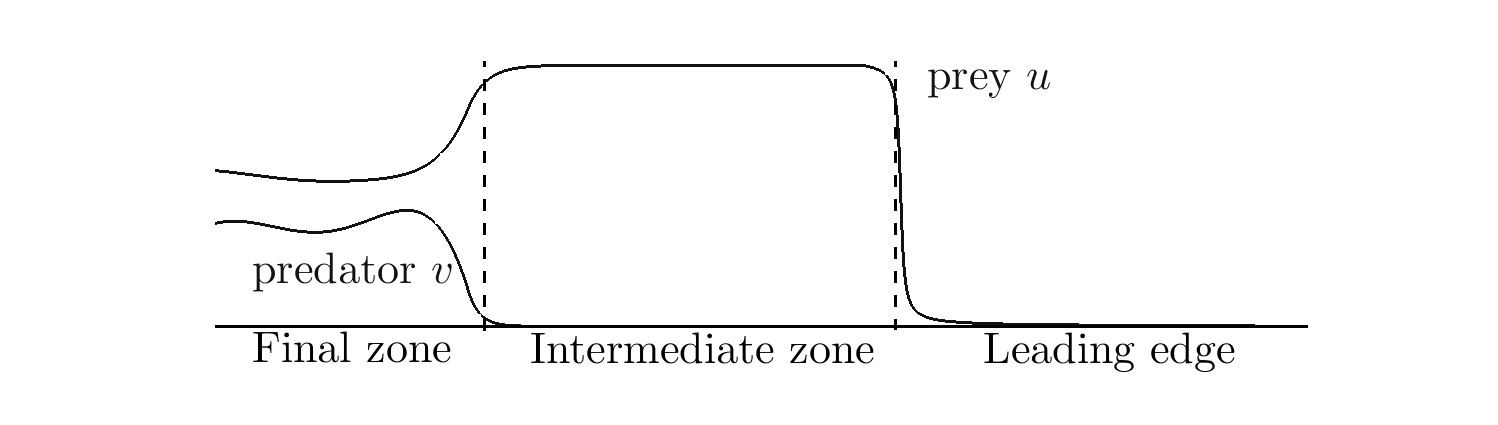}\\
\centerline{Figure 1: Slow predator}\vspace{3mm}\\
\includegraphics[bb=0 0 1557 469, width=0.96\textwidth]{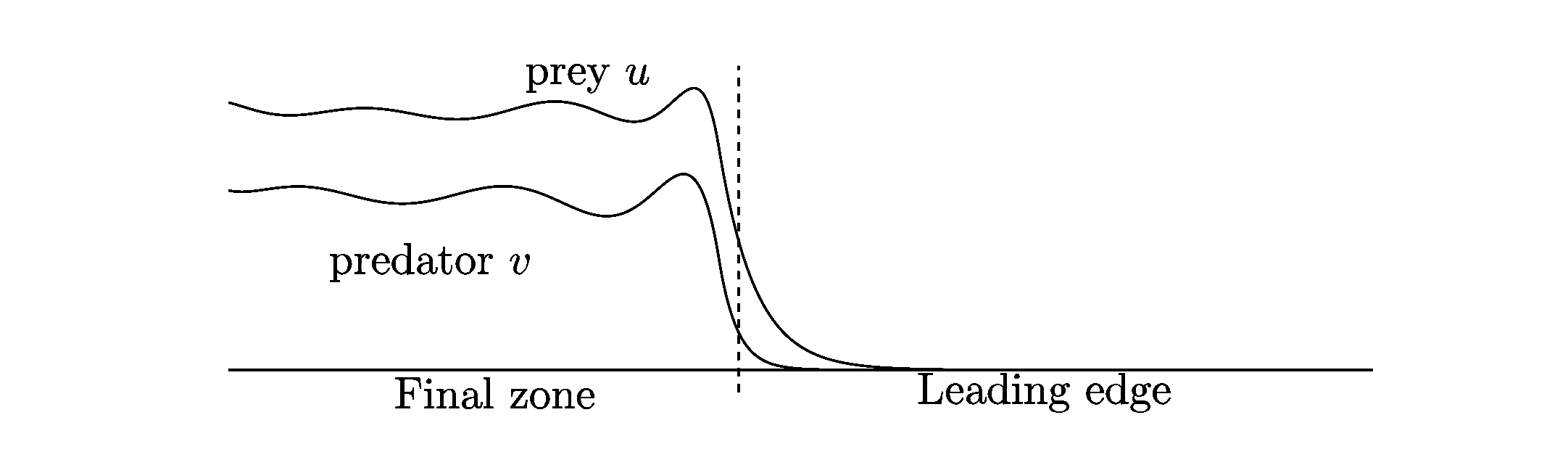}\\
\centerline{Figure 2: Fast predator}
\end{tabular}

\vspace{2ex}

Figures~1 and~2 illustrate the two cases stated in Theorems~\ref{THEO1} and \ref{THEO2}, respectively. In the former case, the prey invades the environment faster than the predator. Therefore, the propagation occurs in two separate steps involving an intermediate equilibrium (namely $u=1$, $v=0$) in between. This multi-front propagation is somewhat reminiscent of the notions of minimal decomposition or propagating terrace for scalar equations~\cite{DGM,FMcL}. On the other hand, in the latter case, the predator's population grows fast enough to always catch up with the prey. Thus both species spread simultaneously, or at least nearly simultaneously.
While there may be some spatial gap between the front of the prey and that of the predator, this gap shall be of order $o(t)$ at most.
Note that the predator cannot spread faster than the prey because it cannot survive without the presence of the prey (see Assumption \ref{ASS-G} \textit{$(b)$} and the subsequent remarks). 

\begin{remark}\label{rmk:KPP0}
Theorems~\ref{THEO1} and~\ref{THEO2} exhaust all the cases. Therefore, if $u_0 \not \equiv 0$ and $v_0 \not \equiv 0$, then propagation of both species occurs. This can be regarded as an analogue of the well-known ``hair-trigger effect" for scalar monostable equation~\cite{AW78}. Moreover, as we see from Theorems~\ref{THEO1} and~\ref{THEO2}, the spreading speed of the prey is always $c^*$, no matter whether the preys are caught up by the predators or not. Note however that the spreading speed being equal to $c^*$ does not mean that the front propagates parallel to the travelling wave of speed~$c^*$. Even in the scalar KPP equation, it is well known that there is a backward phase drift of order $O(\log t)$ from the position $c^* t$. Whether the presence of the predator increases this phase drift or not is yet to be investigated.
\end{remark}
\begin{remark}\label{rmk:KPP}
If we relax the KPP hypothesis on $F$ (Assumption~\ref{ASS-F} \textit{$(b)$} and~\textit{$(c)$}) and assume simply Assumption~\ref{ASS-F} \textit{$(b)$} (monostable nonlinearity), it is well known that the spreading speed $c^*$ for the scalar equation~\eqref{eq-scalar} generally satisfies $c^* \geq 2 \sqrt{d F(0,0)}$ and that this inequality can be strict on certain circumstances. Even in such a situation, as it will be clear from our proof (see Remarks~\ref{rmk:KPP1}, \ref{rmk:KPP2} and \ref{rmk:KPP3} for more details on each zone), Theorem~\ref{THEO1} still holds true as long as $2\sqrt{dF(0,0)} > c^{**}$. This means that any small amount of preys manages to spread outside of the predator's range. On the other hand, in the case $c^* > c^{**} > 2 \sqrt{d F(0,0)}$, more complex dynamics may arise. For example, we suspect that if the initial population of the prey is below a certain threshold, then its density may never become large enough to reach its full nonlinear speed $c^*$, and that the prey may no longer outrace the predator. In such a situation, the propagation of the prey may be slowed down by the effect of predation, making a sharp contrast with what we stated in Remark~\ref{rmk:KPP0}. A similar phenomenon has been observed numerically when the dynamics of the prey population exhibits a (weak or strong) Allee effect (see for instance~\cite{OL}).
%
%
\end{remark}
%

\subsection{Spreading with asymptotics}
In general, how the solutions look in the final zone is not clear, as it depends largely on the dynamics of the underlying system of ordinary differential equations (ODE system for short):
\begin{equation}\label{1.1.ODE}
\begin{split}
&\partial_t u=u F\left(u,v \right),\\
&\partial_t v=v G\left(u,v \right).
\end{split}
\end{equation}
In some parameter range this ODE system has a stable positive equilibrium point, while in other parameter range it may have a limit cycle.
When this ODE system has a globally asymptotically stable equilibrium point $(u^*, v^*)$ in the region $0<u<1$, $v>0$, we expect the solution of the PDE system to converge uniformly in the interior of the final zone to this equilibrium as $t \to +\infty$. We prove this conjecture for the special case where the diffusivity ratio $d$ is equal to $1$, and assuming the existence of a strict Lyapunov function.

\begin{assumption}\label{ass:Lyapunov}
Set $\mathcal O=\{(u,v)\in\R^2:\;0<u<1,\;v>0\}$ and assume that the vector field $\left(F,G\right)$ has a unique singular point $\left(u^*,v^*\right)$ in $\mathcal O$. In other words,
\begin{equation*}
(u,v)\in\mathcal O\text{ and }\left(F,G\right)(u,v)=\left(0,0\right)\;\Rightarrow\;(u,v)=\left(u^*,v^*\right).
\end{equation*}
Assume furthermore that:
\begin{itemize}
\item [$(a)$] There exists a strictly convex function $\Phi:\mathcal O \to \R$ of class $C^2$ that attains its minimum at $\left(u^*,v^*\right)$ and satisfies
\begin{equation}\label{lyap}
(uF(u,v),vG(u,v)) \cdot \nabla \Phi (u,v)\leq 0,\;\forall (u,v)\in\mathcal O.
\end{equation}
\item [$(b)$] The function $\Phi$ is a strict Lyapunov function in the sense that for each $(u_0,v_0)\in \mathcal O$
the following holds true: if $(u,v)$ denotes the solution of \eqref{1.1.ODE} with initial data $(u_0,v_0)$ then
\begin{equation*}
\Phi \left(u(t),v(t)\right)= \Phi \left(u_0,v_0\right),\;\forall t\geq 0 \Rightarrow\;(u_0,v_0)=\left(u^*,v^*\right).
\end{equation*}
\end{itemize}
\end{assumption}
%

Then our result reads as:
\begin{theorem}[Profile of the final zone]\label{THEO.asymptotics}
In addition to Assumptions \ref{ASS-F}, \ref{ASS-G} and \ref{ASS-U}, let the Assumption \ref{ass:Lyapunov} be satisfied.
Let us also assume that $d=1$.
Then for each $c \in [0, \min \{c^{**} ,c^* \})$ one has
\begin{equation}\label{inner-spread}
\lim_{t \to \infty} \sup_{\| x \| \leq ct } \left( \left| u(t,x ) - u^*\right| + \left| v(t,x) - v^*\right| \right) =0,
\end{equation}
where $(u,v)$ is a solution of \eqref{1.1} with nontrivial compactly supported initial data $(u_0,v_0) \in {C}$.
\end{theorem}

Note that in population dynamics, many models admit a Lyapunov function, especially in epidemic models and prey-predator systems. We refer for instance to \cite{CHL} and the references therein for examples of Lyapunov functions in the context of prey-predator systems. In Section~\ref{sec:asymp} we will give two such examples
of prey-predator systems to which the above theorem applies.

Apart from these special cases where \eqref{1.1.ODE} possesses a nice Lyapunov function as in Assumption~\ref{ass:Lyapunov}, the question of finding the behaviour of solutions in the final zone is largely open. One interesting situation is when \eqref{1.1.ODE} has a stable limit cycle. Whether the solution is asymptotically periodic in time in the final zone or not will be the subject of future work.

\subsection{Outline of the paper}\label{sec:outline}

The proof of our main spreading results, namely Theorems~\ref{THEO1} and~\ref{THEO2}, will be performed through several sections dealing with each separate zone shown in Figures~1 and~2. We note that the analysis of the leading edge and the intermediate zone (if it exists) is rather straightforward, since virtually no interaction between the two species occurs in these zones. What is most difficult to analyse is the behaviour of solutions in the final zone, where the two species heavily interact with each other. As we explain later, a large part of the present paper is devoted to the analysis of this final zone.

First, we will show that the prey $u$ cannot propagate faster than the speed~$c^*$, and that the predator cannot propagate faster than the speed $\min \{ c^*,c^{**} \}$. This follows from a simple comparison argument and will be proved in Section~\ref{sec:outer}. Next, we consider the slow predator case (Theorem~\ref{THEO1} \textit{$(ii)$}), in which the intermediate zone appears as shown in Figure~1. Here we need to show that $u$ converges to 1 uniformly in this zone as $t \to +\infty$. As $v$ remains close to~0 in this zone as shown in Section~\ref{sec:outer}, the function $u$ roughly behaves like a solution of the monostable equation \eqref{eq-scalar}; hence one may expect that $u$ converges to the stable state $u=1$ of \eqref{eq-scalar} as $t \to +\infty$. However, as the intermediate zone is moving and expanding, one has to construct suitable sub-solutions carefully. This will be discussed in Section~\ref{sec:interm}.

We will then end the proof of Theorems~\ref{THEO1} and~\ref{THEO2} by investigating the behaviour of the solution in the final zone, which as mentioned above requires the most delicate analysis of all the three zones. This will be the purpose of Section~\ref{sec:inner}. The key point is to prove that $v$ remains uniformly positive in this zone, namely on the ball of growing radius $ct$ for any given $0 \leq c  < \min \{ c^* ,c^{**} \}$. The proof will be divided into three steps. First, we show what we call the ``pointwise weak spreading", which simply states that $v$ does not converge to 0 uniformly as $t \to +\infty$ on the expanding sphere $\| x\|=ct$. The second step consists in proving what we call the ``pointwise spreading". This means that for each direction $e \in S^{N-1}$, $v$ is bounded from below as $t \to +\infty$ along the path $x=cte$ by some constant $\varepsilon >0$, independent of the direction $e \in S^{N-1}$. This crucial step is inspired by some ideas of the dynamical systems theory and more specifically of the uniform persistence theory. We finally conclude the proof by showing that $v$ is asymptotically uniformly positive on the whole ball of growing radius $ct$.

The last two sections will deal with some further results. In Section~\ref{sec:dissip}, we will provide almost optimal conditions for the Assumption~\ref{ASS-U} on weak dissipativity to be satisfied. Lastly, in Section~\ref{sec:asymp}, we study the long time behaviour of solutions in the final zone under the additional Assumption~\ref{ass:Lyapunov} and prove Theorem~\ref{THEO.asymptotics}.

\section{Upper estimates on the spreading speeds}\label{sec:outer}

This section is mainly concerned with the analysis of the leading edge (statements~\textit{$(i)$} of Theorems~\ref{THEO1} and~\ref{THEO2}) and part of the statement~\textit{$(ii)$} of Theorem~\ref{THEO1}. We begin by recalling, for the sake of completeness, the classical spreading result for monostable scalar equations from Aronson and Weinberger~\cite{AW78}:
\begin{proposition}[Spreading for the scalar monostable equation~\cite{AW78}]\label{LE1} \hspace{1mm} 

\noindent Consider the so-called monostable equation
\begin{equation}\label{eqn:monostable}
\left(\partial_t -d\Delta\right)u(t,x)=u(t,x) f(u(t,x)),\;t>0,\;x\in\R^N,
\end{equation}
wherein $f$ a $C^1$ function satisfying
$$f(1) = 0 \mbox{ and } f(u) >0 \mbox{ for all } u \in [0,1).$$
Then there exists $c^* (d,f)$ satisfying $c^* (d,f) \geq 2 \sqrt{df(0)}$ such that, for any nontrivial compactly supported and continuous function~$0 \leq \varphi \leq 1$, the solution $u\equiv u(t,x;\varphi)$ of \eqref{eqn:monostable} with initial data $\varphi$ satisfies:
\begin{equation*}
\lim_{t\to\infty} \sup_{\|x\|\geq ct} u(t,x;\varphi)=0,\;\;\forall c>c^*(d,f),
\end{equation*}
and
\begin{equation*}
\lim_{t\to\infty} \sup_{\|x\|\leq ct} \left|1-u(t,x;\varphi)\right|=0,\;\;\forall c\in \left(0,c^* (d,f)\right).
\end{equation*}
Furthermore, $c^* (d,f)$ coincides with the minimal speed of travelling waves connecting 0 to 1 of the one-dimensional equation $(\partial_t - d \partial_x^2)u = uf(u)$.

If in addition $f(u) \leq f(0)$ for all $u \in [0,1]$, then $c^*(d,f) = 2 \sqrt{df(0)}$.
\end{proposition}

This result directly applies to the equation~\eqref{eq-scalar}, hence we can refer to $c^*$ defined by \eqref{speed_prey} as the spreading speed of solutions of \eqref{eq-scalar}. The argument in~\cite{AW78} also applies to the equation~\eqref{eq-scalar-bis}, except that convergence to 1 in the final zone (namely in any ball $\| x \| \leq ct$ with $c < c^{**}$) must be replaced by uniform positivity.

\begin{remark}\label{rmk:mono+}
Let us recall the notion of travelling waves for equations of the form~\eqref{eqn:monostable} or its one-dimensional version
\begin{equation}\label{eq:mono+}
(\partial_t -  d \partial_x^2) u = u f(u) , \; t>0 , \; x \in \R.
\end{equation}
A solution~$u$ of~\eqref{eq:mono+} is called a travelling wave connecting $0$ to $p$, where $p$ is a positive zero of $f$, if it is written in the form $u(t,x) = U (x-ct)$ with some constant~$c$ and a function~$U(z)$ satisfying $U(-\infty)=p$, $U(+\infty)=0$, $0< U (z) < p$ ($z \in \mathbb{R}$). The constant~$c$ is called the speed of this travelling wave, and $U$ its profile function. As it is easily seen, $U$ satisfies
$$d U'' + cU' + Uf(U) = 0.$$
In the multidimensional equation~\eqref{eqn:monostable}, for any unit vector~$\nu \in \R^N$, the function of the form $U(x\cdot \nu - ct)$, where $U(z)$ is the above profile function of a one-dimensional travelling wave, is a solution of~\eqref{eqn:monostable} (where $f(p)=0$), and is called a planar wave solution. It is known that if~$f$ is of the monostable type, then there exists the minimal travelling wave speed (see~\cite{AW78}).\end{remark}
\begin{remark}\label{rmk:KPP1}
Here we recalled the spreading property for the general monostable equation. The reason is that the arguments in this section, as well as those in Sections~\ref{sec:interm} and~\ref{sec:inner}, remain valid if one assumes equation~\eqref{eq-scalar} to be monostable, instead of the KPP hypothesis \textit{$(c)$} of Assumption~\ref{ASS-F}. In that case, $c^*$ is understood to denote the spreading speed of solutions of \eqref{eq-scalar} (for compactly supported initial data). Hence in general we have $c^* \geq 2\sqrt{dF(0,0)}$ instead of~\eqref{speed_prey}.
\end{remark}

We now begin the proof of Theorems~\ref{THEO1} and~\ref{THEO2}, with some simple upper estimates on the spreading speeds. Let $\overline{u}$ be the solution of 
$$\partial_t \overline{u} = d \Delta \overline{u} + \overline{u} F(\overline{u},0),$$
associated with the same initial data $u_0$ as given in the assumptions of Theorems~\ref{THEO1} and~\ref{THEO2}. Because $v \geq 0$ and $F(u,v)$ is nonincreasing with respect to $v$, we immediately get that $\overline{u}$ is a super-solution for the $u$-equation in system~\eqref{1.1}. Applying the comparison principle and thanks to Proposition~\ref{LE1}, one gets that
$$\forall c > c^*, \ \lim_{t\to +\infty} \sup_{\| x\| \geq ct} u (t,x)  \leq\lim_{t\to +\infty} \sup_{\| x\| \geq ct} \overline{u} (t,x) =0 .$$
Since $u$ is nonnegative, this already proves statement~\textit{$(i)$} in Theorem~\ref{THEO1}, as well as half of statement~\textit{$(i)$} in Theorem~\ref{THEO2}.\\

Furthermore, we know that $u \leq 1$, so that from Assumption~\ref{ASS-G}, the function
\begin{equation}\label{sursol_11}
\overline{v}_{1,e} := A e^{-\sqrt{G(1,0)} (x \cdot e -c^{**} t)},
\end{equation}
satisfies, for any $e \in S^{N-1}$ and $A >0$:
\begin{eqnarray*}
&& \partial_t \overline{v}_{1,e} - \Delta \overline{v}_{1,e} - \overline{v}_{1,e} G (u, \overline{v}_{1,e})\\
& \geq & \partial_t \overline{v}_{1,e} - \Delta \overline{v}_ {1,e} - \overline{v}_{1,e} G (1, 0)\\
& = & (c^{**} \sqrt{G(1,0)} - 2G(1,0)) \times \overline{v}_{1,e}\\
& = & 0,
\end{eqnarray*}
where we have used \eqref{speed_predator} defining the value $c^{**}$.

Since $v_0$ is compactly supported and bounded, one can choose $A$ large enough so that for any $e \in S^{N-1}$, the inequality $v_0 \leq \overline{v}_{1,e}$ holds. We can then conclude from the parabolic comparison principle that
$$\forall c > c^{**}, \ \lim_{t \to +\infty} \sup_{\| x \| \geq ct } v (t,x) \leq \lim_{t \to +\infty} \; \sup_{\| x \| \geq ct } \; \inf_{e \in S^{N-1}}  \overline{v}_{1,e} (t,x)= 0.$$
The function $v$ being nonnegative, half of statement~\textit{$(ii)$} in Theorem~\ref{THEO1} is now proved.\\

Now, let us show that $v$ cannot spread faster than $c^*$ either. This stems from the fact that $G(0,0) < 0$, which means that $v$ cannot spread by itself outside of the range of $u$. In particular, for any $c >c^*$, there is some $\eta >0$ so that $$2\sqrt{G(\eta,0)} < c.$$ There then exists some $0< \lambda < \sqrt{G(\eta,0)}$ such that
$$\lambda^2 - c \lambda + G(\eta,0) =0.$$
From the above upper estimate on the spreading of $u$, there also exists some $t_1 >0$ such that, for any $t \geq t_1$,
$$\sup_{t \geq t_1} \sup_{ \| x \| \geq ct } u (t,x) \leq \eta .$$
Then, for any constant $B >0$, the function
$$\overline{v}_{2,e}  := B e^{-\lambda (x \cdot e - ct)}$$
satisfies
\begin{eqnarray*}
&&  \partial_t \overline{v}_{2,e} - \Delta \overline{v}_ {2,e} - \overline{v}_{2,e} G (\eta, 0)\\
& = & (c\lambda - \lambda^2 - G(\eta,0)) \times \overline{v}_{2,e}\\
& = & 0,
\end{eqnarray*}
thus it is a super-solution of the $v$-equation in system~\eqref{1.1} for any $t \geq t_1$ and $\|x \| \geq ct $. 

Moreover, on the one hand, thanks to Assumption~\ref{ASS-U} on the weak dissipativity, there exists~$B$ large enough so that 
$$v(t,x) \leq B \leq \overline{v}_{2,e} (t,x),$$
for any $e \in S^{N-1}$, $t \geq t_1$ and $\|x\| = ct$. 

On the other hand, since $\lambda < \sqrt{G(\eta,0)} < \sqrt{G(1,0)}$, by choosing~$B$ large enough, we may also assume that
$$v(t_1,x) \leq \overline{v}_{1,e} (t_1,x) \leq \overline{v}_{2,e} (t_1,x),$$
for any $e \in S^{N-1}$ and all $x \in \R^N$. We can now apply once again the comparison principle to conclude that
$$\forall c' > c, \ \lim_{t \to +\infty} \sup_{\| x \| \geq c' t } v(t,x) =0.$$
Since $c$ can be chosen arbitrarily close to $c^*$, this implies as announced that~$v$ does not spread faster than the speed $c^*$, which completes the proof of statement~\textit{$(i)$} of Theorem~\ref{THEO2}.

\section{Spreading of $u$ outside of the predator's range}\label{sec:interm}

In this section, we complete the proof of statement~\textit{$(ii)$} in Theorem~\ref{THEO1}, which concerns the intermediate zone. To do so, we will assume the following instead of assuming $c^* > c^{**}$:
\begin{equation}\label{eqn:inq_speeds}
2\sqrt{d F(0,0)} > c^{**} = 2\sqrt{G(1,0)}.
\end{equation}
\begin{remark}\label{rmk:KPP2}
Under the KPP hypothesis on $F$ (see Assumption~\ref{ASS-F} \textit{$(c)$}), the inequality \eqref{eqn:inq_speeds} reads as $c^* > c^{**}$. However, as we mentioned in Remark~\ref{rmk:KPP}, the KPP assumption is actually not needed here. Instead, we will only assume throughout this section that \eqref{eq-scalar} is of the monostable type (namely $F(u,0) >0$ for all $0 \leq u <1$) and that the inequality~\eqref{eqn:inq_speeds} above holds.
\end{remark}

As we have already shown that $v$ converges to 0 outside of the ball $\| x\| \leq ct$ for any $c > c^{**}$, it only remains to check that for any $c^{**} < c_1 < c_2 < 2\sqrt{d F(0,0)}$, one has
$$\lim_{t\to \infty} \sup_{c_1 t \leq \| x \| \leq c_2 t } 1 - u(t,x) = 0.$$
Note that this means $\limsup |1-u| =0$ since $1-u \geq 0$.

In the following, we fix any such $c_1$ and $c_2$. The main idea is to use the fact that~$v$ tends to 0 on this part of the domain or, in other words, that those moving frames are outside of the predator's range. Therefore, we expect $u$ to behave like a solution of \eqref{eq-scalar}.\\

Before we begin the proof, let us introduce the following equation, where $\varepsilon$ is a positive parameter:
\begin{equation}\label{eq-scalar-eps}
\partial_t u = d \Delta u + u F(u,\varepsilon).
\end{equation}
Provided that $\varepsilon$ is small enough, this equation is of the standard monostable type, just as \eqref{eq-scalar}, in the sense of~\eqref{eqn:monostable_def} below. Thus the following lemma holds:
\begin{lemma}\label{lem:eq-scalar-eps}
For any small enough $\varepsilon >0$, there exist some constants $0<p_\varepsilon \leq 1$ and $c_\varepsilon >0$ such that 
\begin{equation}\label{eqn:monostable_def}
F(p_\varepsilon,\varepsilon) = 0, \ \mbox{ and } \ F(u,\varepsilon) > 0 \ \mbox{ for any } \ 0 \leq u < p_\varepsilon,
\end{equation}
and
\begin{itemize}
\item [$(i)$] equation \eqref{eq-scalar-eps} admits travelling waves connecting 0 to $p_\varepsilon$ with speed $c$ if and only if $c \geq c_\varepsilon$ (see Remark~\ref{rmk:mono+});
\item [$(ii)$] for any nontrivial, compactly supported and continuous initial data $0 \leq \varphi \leq p_\varepsilon$, the associated solution $u (t,x)$ of \eqref{eq-scalar-eps} spreads with speed $c_\varepsilon$ in the following sense:
\begin{equation*}
\lim_{t\to\infty} \sup_{\|x\|\geq ct} u(t,x)=0,\;\;\forall c>c_\varepsilon,
\end{equation*}
and
\begin{equation*}
\lim_{t\to\infty} \sup_{\|x\|\leq ct} \left|p_\varepsilon -u(t,x)\right|=0,\;\;\forall c\in \left(0,c_\varepsilon \right).
\end{equation*}
\end{itemize}
Furthermore, $c_\varepsilon$ and $p_\varepsilon$ are nonincreasing with respect to $\varepsilon$ and tend respectively to $c^*$ and 1 as $\varepsilon \to 0$.
\end{lemma}

Although its statement may seem a bit intricate, this lemma simply says that the dynamics of equation~\eqref{eq-scalar-eps} is close to that of equation~\eqref{eq-scalar}.\\

\begin{proof} Define
$$p_\varepsilon := \min \{ u \geq 0 \; | \ F(u,\varepsilon) \leq 0 \}.$$
From the monotonicity of $F$ with respect to~$v$, it is clear that $p_\varepsilon \leq 1$ is well defined. Moreover, since $F$ is regular and satisfies $F(0,0)>0$, we have 
$$p_\varepsilon >0, \ F(p_\varepsilon , \varepsilon) =0, \ \mbox{and } F(u,\varepsilon)>0 \mbox{ for any } 0 \leq u < p_\varepsilon,$$
provided that~$\varepsilon$ is small enough.

In other words, \eqref{eq-scalar-eps} is of the monostable type in the range $[0,p_\varepsilon]$ with $0$, $p_\varepsilon$ being the equilibrium states. Let $c_\varepsilon >0$ denote the spreading speed for solutions of~\eqref{eq-scalar-eps} with nontrivial compactly supported initial data (see Proposition~\ref{LE1}).

From the fact that $F$ is nonincreasing with respect to~$v$, one can easily check that $c_\varepsilon$ and $p_\varepsilon$ are nonincreasing with respect to $\varepsilon$ and that they are bounded from above by respectively $c^*$ and $1$. In particular, $c_\varepsilon$ and $p_\varepsilon$ converge respectively to $c_0 \leq c^*$ and $p_0 \leq 1$. From the regularity of $F$, one has that $F(p_0, 0)=0$, thus $p_0 =1$. 

It only remains to check that $c_0 = c^*$. As mentioned before (see Proposition~\ref{LE1} and Remark~\ref{rmk:mono+}), there exists some travelling wave solution of equation~\eqref{eq-scalar-eps} with speed $c_\varepsilon$ satisfying $U_\varepsilon (z) \in (0,p_\varepsilon) $ and
$$d U_\varepsilon '' + c_\varepsilon U_\varepsilon ' + U_\varepsilon  F (U_\varepsilon ,\varepsilon) =0,$$
along with the limiting conditions
$$U_\varepsilon (-\infty) = p_\varepsilon ,\  \mbox{ and } \ U_\varepsilon (+\infty) =0.$$
This travelling wave is also known to be monotonically decreasing~\cite{AW78}, and up to some shift, it can be assumed to satisfy $U_\varepsilon (0) = \frac{p_\varepsilon}{2}$.

The sequence $U_\varepsilon$ is uniformly bounded, hence by standard elliptic estimates, it converges locally uniformly (and in fact, strongly in $C^1_{loc} (\R)$ and weakly in $W^{2,p}_{loc} (\R)$ for any $p \geq 1$), up to extraction of a subsequence as $\varepsilon \to 0$, to some nonincreasing function $U_0 (z) \in (0,1)$ solution of
$$d U_0 '' + c_0 U_0 ' + U_0  F (U_0 ,0) =0,$$
satisfying also $U_0 (0) = 1/2$. It is clear that the following limits exist
$$1 \geq U_+ := \lim_{z \to -\infty} U_0 (z) \geq \frac{1}{2} \geq U_- := \lim_{z \to +\infty} U_0 (z) \geq 0.$$
Moreover, $U_+$ and $U_-$ have to be steady states, that is
$$U_\pm F (U_\pm ,0) =0.$$
The only possibility is that $U_+ = 1 $ and $U_- = 0$. This means that $U_0$ is a travelling wave solution of the monostable equation~\eqref{eq-scalar} with speed $c_0$. Using again the classical result in~\cite{AW78}, the spreading speed $c^*$ which we defined in Proposition~\ref{LE1} is also the minimal speed of travelling wave solutions of~\eqref{eq-scalar} connecting 0 to 1. This now implies that $c_0 \geq c^*$, hence $c_0 = c^*$. This concludes the proof of Lemma~\ref{lem:eq-scalar-eps}.
\end{proof}\\

Our goal is now to compare solutions of \eqref{1.1} with solutions of \eqref{eq-scalar-eps}, using our upper spreading estimate on the $v$-component. The proof consists of two main steps. In the first part, we show that $u$ is always able to reach some favourable region, outside of the predator's range, where it is able to grow close to the positive steady state~$1$. In the second part, we prove that $u$ spreads from this favourable region at speed $c^*$ in the outward direction, using the classical construction of sub-solutions for the monostable equation \eqref{eq-scalar-eps}. 

In this scheme, the two steps are important because, in the general monostable (non-KPP) case, such sub-solutions cannot be chosen arbitrarily small.

\begin{lemma}\label{spread_lem1}
For any $c \in \left(c^{**},2\sqrt{dF(0,0)}\right)$, the solution $(u,v)$ of \eqref{1.1} satisfies
$$\lim_{t \to +\infty}   u (t,x+cte)  =1,$$
where the convergence holds locally uniformly with respect to~$x$ and uniformly with respect to $e \in S^{N-1}$.
\end{lemma}
\begin{proof}
First choose any $\varepsilon>0$ such that the inequality $ 2 \sqrt{dF(0,\varepsilon)} > c$ holds, thanks to the continuity of~$F$ and our choice of $c \in (c^{**},2\sqrt{dF(0,0)})$. We recall from the previous section, where we bounded $v$ from above by exponential functions moving with speed $c^{**}$ and defined in~\eqref{sursol_11}, that there exists some $X_\varepsilon >0$ such that
\begin{equation}\label{spread_lem1_eq1}
v (t,x) \leq \varepsilon, \text{ for all } (t,x) \mbox{ such that } \|x \| \geq X_\varepsilon + c^{**} t.
\end{equation}
The proof of Lemma~\ref{spread_lem1} will be divided into three steps.

\paragraph{Step 1:} We choose some $c' \in (c,2\sqrt{dF(0,\varepsilon)})$ and claim that
\begin{equation}\label{lem1_step1}
\exists R_1 >0, \delta_1 >0,  \ \ \liminf_{t \to +\infty} \inf_{e \in S^{N-1}} \inf_{x \in B_{R_1} } u \left( \frac{ct}{c'},x+(ct+X_\varepsilon+ 2R_1 )e \right) > \delta_1,
\end{equation}
where $B_R$ denotes the ball of radius~$R$ centred at the origin. 

Let $\phi_{2R}$ be the principal eigenfunction, normalized so that $\| \phi_{2R} \|_\infty =1$, of the Laplace operator on the ball $B_{2R}$ with Dirichlet boundary conditions, that is solving
$$
\left\{
\begin{array}{ll}
\Delta \phi_{2R} = \mu_{2R} \phi_{2R} & \mbox{ in } B_{2R},\vspace{3pt}\\
\phi_{2R} >0 & \mbox{ in } B_{2R},\vspace{3pt}\\
\phi_{2R} =0& \mbox{ on } \partial B_{2R}.
\end{array}
\right.
$$
It is known that $\mu_{2R} <0$ for any $R>0$, and tends to $0$ as $R \to +\infty$. In what follows, we will extend $\phi_{2R}$ on the whole space by setting $\phi_{2R} (x)=0$ if $\| x \| > 2 R$.

Recall that $c'  \in \left( c, 2\sqrt{dF (0,\varepsilon)} \right)$, thus for any $e \in S^{N-1}$, the function 
$$\displaystyle \underline{u} (t,x) :=\displaystyle  \eta e^{-\frac{c'}{2d} (x \cdot e -c't) } \phi_{2R} (x-(c't + X_\varepsilon +2R)e)$$
is a sub-solution of the $u$-equation in system~\eqref{1.1}. Indeed, one has:
\begin{equation*}
\begin{split}
\partial_t \underline{u} - d \Delta \underline{u} - \underline{u} F (\underline{u},v) &\leq   \partial_t \underline{u} - d \Delta \underline{u} - \underline{u} F (\underline{u},\varepsilon) \\
& \leq  \left( \frac{c'^2}{4d} - d\mu_{2R} \right) \underline{u} - \underline{u} F (\underline{u},\varepsilon) \\
& \leq   \left( \frac{c'^2}{4d} - d\mu_{2R} - F (0,\varepsilon) \right) \underline{u} + \underline{u}^2 \| \partial_u F \|_\infty \\
& \leq  0,
\end{split}
\end{equation*}
where the computation is performed on the support of $\underline{u}$, which is included in the set $\{ (t,x) \in \R^{N+1} \; | \ \|x \| \geq X_\varepsilon + c^{**}t \}$ where $v \leq \varepsilon$ thanks to~\eqref{spread_lem1_eq1}, and the last inequality is true provided that $R$ is large enough and $\eta$ small enough so that 
$$ -d \mu_{2R} < F (0,\varepsilon) - \frac{c'^2}{4d},$$
$$ \| \partial_u F \|_\infty \underline{u} \leq  F (0,\varepsilon) - \frac{c'^2}{4d} + d \mu_{2R}.$$
Choosing $\eta$ even smaller if necessary and recalling that $u$ is positive for any positive time, we can also assume that $u(1,x) \geq \underline{u} (1,x)$ for all $x \in \R^N$. Hence, by the comparison principle, we get that 
$$u \left( \frac{ct}{c'},x \right) \geq  \underline{u} \left( \frac{ct}{c'},x \right) = \underline{u} (0,x-cte),$$
for any $t \geq \frac{c'}{c}$, $x \in \R^N$ and $e \in S^{N-1}$. The claim~\eqref{lem1_step1} easily follows by taking $R_1 = R$ and $$\delta_1 =  \eta e^{-\frac{c'}{2d} (X_\varepsilon + 3R)} \inf_{x \in B_R} \phi_{2R} (x) >0.$$

\paragraph{Step 2:} We now claim that
\begin{equation}\label{lem1_step2}
\exists R_2 >0, \delta_2 >0,  \ \ \liminf_{t \to +\infty} \inf_{e \in S^{N-1}} \inf_{(s,x)\in D_t} u \left( s,x+(ct+X_\varepsilon+ 2R_1)e \right) > \delta_2,
\end{equation}
wherein $D_t$ is defined by
$$D_t := \left\{ (s,x) \; | \ \frac{ct}{c'} \leq s \leq t \mbox{ and } x \in B_{R_2} \right\}.$$
In other words, we now bound $u$ away from zero not only in the moving frames with speed $c$ in all directions, but also around the whole sphere $\{ \|x \| = ct \}$ for a large range of times $s \in (\frac{ct}{c'} , t)$.

Note that in the previous step, $R_1$ can be chosen arbitrarily large. Hence, proceeding as above, we can also construct a stationary sub-solution of \eqref{eq-scalar-eps} of the form $\underline{\underline{u}} (x) := \eta ' \phi_{R_1 /2} (x)$ by increasing $R_1$ if necessary. Choosing $\eta' < \delta_1$, we get from \eqref{lem1_step1} that
$$u \left( \frac{ct}{c'},x + (ct + X_\varepsilon + 2 R_1) e \right) \geq \underline{\underline{u}} (x),$$
for any large $t$, $x \in \R^N$ and $e\in S^{N-1}$. Using the comparison principle, we conclude that \eqref{lem1_step2} holds true with $R_2 = \frac{R_1}{3}$ and $$\delta_2 =\eta' \inf_{x \in B_{R_2}} \phi_{R_1 /2} (x)>0.$$

\paragraph{Step 3:} We will now complete the proof of Lemma~\ref{spread_lem1}. Thanks to standard parabolic estimates and \eqref{spread_lem1_eq1}, one can extract a subsequence $t_n \to +\infty$ and $e_n \to e_\infty$ in $S^{N-1}$ such that
$$u_n (t,x) := u (t_n +t, x + ct_n e_n)$$
converges locally uniformly to $u_\infty (t,x)$. Using \eqref{spread_lem1_eq1} again, together with the monotonicity of $F$ with respect to $v$, we see that $u_\infty$ is a super-solution of
$$\partial_t u_\infty - d \Delta u_\infty - u_\infty F(u_\infty ,\varepsilon) \geq 0,$$
for all $t\in \R$ and $x \in \R^N$. Moreover, by \eqref{lem1_step2}, we know that for all $t \leq 0$,
$$\inf_{x \in B_{R_2} } u_\infty (t,x+(X_\varepsilon+2 R_1)e_\infty ) \geq \delta_2.$$
By the comparison principle, it is clear that for all $t \geq 0$,
$$u_\infty (0,x) \geq \tilde{u} (t,x),\;\forall x\in\R^N.$$
Here $\tilde{u}$ denotes the solution of
$$\partial_t \tilde{u} - d \Delta \tilde{u} - \tilde{u} F(\tilde{u} ,\varepsilon) = 0,$$
with initial data $\tilde u(0,x)=\delta_2 \chi_{B_{R_2}} (x-(X_\varepsilon +2 R_1)e_\infty)$, and $\chi$ is the characteristic function.

From Lemma~\ref{lem:eq-scalar-eps}, we know that $\tilde{u} (t,\cdot)$ converges locally uniformly to $p_\varepsilon$ as $t \to +\infty$, thus $u_\infty (0,x) \geq p_\varepsilon$ for all $x \in \R^N$ (in fact, the same argument shows that this is also true at any other time). Recalling the construction of $u_\infty$, we get that
$$\liminf_{t \to +\infty} u(t, x+cte ) \geq p_\varepsilon,$$
locally uniformly with respect to $x \in \R^N$, and uniformly with respect to $e \in S^{N-1}$. As $\varepsilon$ was chosen arbitrarily small, we can pass to the limit as $\varepsilon \to 0$, which completes the proof of Lemma~\ref{spread_lem1}.\end{proof}

\begin{remark} We point out that the intended conclusion, namely \textit{$(ii)$} of Theorem~\ref{THEO1}, could easily follow in the KPP case from the  lemma above. The second part below extends our spreading result for the intermediate zone to more general nonlinearities, and we include it for the sake of completeness.\end{remark}

In order to complete the proof of Theorem~\ref{THEO1} \textit{$(ii)$}, we need to show that~$u$ converges to 1 uniformly in the region $c_1 t \leq \| x \| \leq c_2 t$ as $t \to \infty$. This will be done by constructing an appropriate set of sub-solutions with compactly supported initial data. However, as we are dealing with a general monostable nonlinearity, finding a good sub-solution is not as simple as in the KPP case. Now that we have found a large enough zone where $u$ takes relatively large values, we can use some slight modification of the argument of the paper of Aronson and Weinberger~\cite{AW78} to construct two families of sub-solutions of the monostable equation~\eqref{eq-scalar-eps}, one moving with speed~$c_1$ and the other with speed~$c_2$.

First, the following lemma holds:
\begin{lemma}\label{lem:AW}
For any $0 < c <c_\varepsilon$, there exists some $\alpha (c) < p_\varepsilon$ such that, for any $\alpha \geq \alpha (c)$, the solution~$q_c$ of
\begin{equation}\label{eq:ode}
dq'' + cq'+ qF(q,\varepsilon)=0 , \  q'(0) = 0 \ \mbox{ and } q (0 ) = \alpha,
\end{equation}
satisfies $q_c $(b)$ = 0$ for some $b >0$, as well as $q_c ' < 0$ on $(0,b]$.
\end{lemma}

Those functions can be obtained by some standard phase plane analysis. Roughly speaking, trajectories from $(q=p_\varepsilon,q'=0)$ either enter the $\{q <0\}$ half-plane if $c< c_\varepsilon$, or converge to $(q=0,q'=0)$ when $c \geq c_\varepsilon$, in which case they are travelling wave solutions. Those trajectories are of course approached by solutions of \eqref{eq:ode} with $\alpha$ close enough to $p_\varepsilon$. We refer the reader to~\cite{AW78} (in particular Lemma~4.3) for the details.\\

We are now in a position to complete the proof of Theorem~\ref{THEO1} \textit{$(ii)$} without the KPP assumption.\\
%

\begin{proof}[Proof of Theorem~\ref{THEO1} \textit{$(ii)$}] Let $c \in (c_2,c_\varepsilon)$ be given (recall that $c^{**} < c_1 < c_2 < c^*$), and let us consider the two family of functions
$$\underline{u}_{c_i}  (t,x;e)= 
\left\{
\begin{array}{ll}
\alpha &  \mbox{ if } \ \|x - c_i te \| < \rho ,\vspace{3pt}\\
q_{c} ( \| x - c_i te \| - \rho ) & \mbox{ if } \  \rho \leq \| x - c_i te \| \leq \rho + b,\vspace{3pt}\\
0 &   \mbox{ if } \ \|x -c_i te \| > \rho + b,\vspace{3pt}
\end{array}
\right.
$$
where $i=1,2$, $e \in S^{N-1}$, $1 > \alpha \geq \alpha (c)$ is provided by Lemma~\ref{lem:AW} and $\rho>0$ is to be chosen large enough later.  

From Lemma~\ref{spread_lem1}, there exists some $T>0$ such that
$$u (T,x) > \alpha,$$
for all $\|x \| \in (cT - \rho - b, cT +\rho +b)$. Thus, denoting $T_i = \frac{c}{c_i} T$, we get
$$u (T,x) > \underline{u}_{c_i} (T_i,x;e),$$
for any $i=1,2$ and $e \in S^{N-1}$. Recall that for any $\varepsilon >0$, there exists some $X>0$ such that
$$v (t,x) \leq \varepsilon, \ \forall (t,x) \mbox{ such that } \|x \| \geq X + c^{**} t.$$
In particular, the solution $u (t,x)$ of \eqref{1.1} is a super-solution of \eqref{eq-scalar-eps} in the same domain, in which is included the support of the functions $\underline{u}_{c_i} (\cdot,\cdot;e)$ for any $t \geq T_i$, up to increasing $T$.

Let us now check that these functions $\underline{u}_{c_i}$ are sub-solutions for~\eqref{eq-scalar-eps}. This is trivial for $\|x -c_i te\| < \rho$ and $\|x - c_i t e\| > \rho + b$. Lastly, for $\rho < \|x -c_i te\| < \rho +b$, one can compute that
\begin{eqnarray*}
&& \partial_t \underline{u}_{c_i} - d \Delta \underline{u}_{c_i} - \underline{u}_{c_i} F(\underline{u}_{c_i},\varepsilon)\\
& = &  -q'_{c} (\| x - c_i te\| - \rho) \times \frac{ c_i (x-c_i te) \cdot e }{\|x-c_ite\|}  -q_{c} (\| x - c_i te\| - \rho) F(q_{c} (\| x - c_i te\| - \rho),\varepsilon)\\
&&  - d \left( q''_{c} (\| x - c_ite\| - \rho ) + q'_{c} (\| x - c_ite\| - \rho) \times \frac{N-1}{\|x - c_i te\|} \right) \\
& = &  q'_{c} (\| x - c_ite\| - \rho) \times \left( c - c_i \frac{(x-c_ite)\cdot e}{\|x - c_ite\|} - d \frac{N-1}{\|x -c_ite\|} \right)\\
& \leq & q'_{c} (\| x - c_ite\| - \rho) \left( c -  c_i - d \frac{N-1}{\|x -c_ite\|} \right)\\
& < & 0,
\end{eqnarray*}
provided that $\rho$ is large enough (recall that $q'_c < 0$).

This is sufficient to apply a comparison principle and get that $u(t+T,x) > \underline{u}_{c_i} (t+T_i,x;e)$ for any $e$, $i=1,2$ and $t>0$. Indeed, proceed by contradiction and assume that for some $e$ and $i$, there exists some $t_0 > 0$ such that for all $0 \leq t < t_0$,
$$u (t+T,x) > \underline{u}_{c_i} (t+ T_i,x;e), $$
and at time $t_0$, there exists $x_0 \in \R^N$ such that
$$u(t_0 +T ,x_0) = \underline{u}_{c_i} (t_0 +T_i,x_0;e).$$
From the positivity of $u$ and the strong maximum principle, one immediately gets that $\|x_0 - c_i (t_0 +T_i) e\| = \rho$. Then, applying parabolic Hopf lemma inside the ball $\{ \| x - c_i (t+T_i) e \| < \rho \}$, we get that
$$\nabla u (t_0,x_0) \cdot \frac{x_0 - c_i (t_0 + T_i) e}{\|x_0 - c_i (t_0 + T_i)\|} <0.$$
Applying it outside of the same ball, 
$$\nabla u (t_0,x_0) \cdot \frac{x_0 - c_i (t_0 + T_i) e}{\|x_0 - c_i (t_0 + T_i)\|} > q'_c (0)=0.$$
Having reached this contradiction, we can now conclude that for all positive time $t$,
\begin{equation}\label{eq:33}u (t+T,x) > \underline{u}_{c_i} (t+T_i,x;e).
\end{equation}
Define now the subdomain 
$$Q := \{ \{t \} \times Q_t \}_{t \geq T}, \mbox{ where } Q_t :=\{ x \; | \  cT +c_1 (t-T) \leq  \|x\| \leq cT + c_2 (t-T) \}.$$
It immediately follows from \eqref{eq:33} that $u \geq \alpha$ on $\partial Q$. Recall once again \eqref{spread_lem1_eq1}, so that $\alpha$ is a sub-solution of the equation satisfied by $u$ in $Q$. Using some extended maximum principle (see for instance Lemma~2.2 in \cite{BHNa}), it follows that 
$$\forall (t,x) \in Q, \ u (t,x) \geq \alpha.$$
Thus one obtains
$$\liminf_{t \to \infty} \inf_{c_1 t + (c-c_1)T \leq \|x\| \leq c_2 t + (c-c_2) T} u (t,x) \geq \alpha,$$
and, by slightly modifying $c_1$ and $c_2$ and without loss of generality, one can easily get rid of the extra $(c-c_i)T$ so that
$$\liminf_{t \to \infty} \inf_{c_1 t \leq \|x\| \leq c_2 t} u (t,x) \geq \alpha.$$
Since $\varepsilon$ could be chosen arbitrarily small, and $\alpha$ arbitrarily close to $p_\varepsilon$, this yields
$$\liminf_{t \to \infty} \inf_{c_1 t \leq \|x\| \leq c_2 t } u (t,x) \geq 1.$$
This completes the proof of statement~\textit{$(ii)$} in Theorem~\ref{THEO1}.\end{proof}

\section{Lower estimates on the spreading speed}\label{sec:inner}

It now remains to prove statements~\textit{$(iii)$} of Theorem~\ref{THEO1} and~\textit{$(ii)$} of Theorem~\ref{THEO2} concerning the final zone (see Figures~1 and~2). These two statements will be proved simultaneously. Throughout this section, we will assume that $c^0$ is an arbitrarily fixed constant such that
\begin{equation}\label{eq:speed_inner}
 0 \leq c^0 < \min \{ c^*, c^{**} \}.
\end{equation}
\begin{remark}\label{rmk:KPP3}
As in the previous sections, one may weaken the KPP hypothesis on $F$. We will again only assume that equation~\eqref{eq-scalar} is monostable and, instead of inequality~\eqref{eq:speed_inner}, that
$$ 0 \leq c^0 < \min \{ 2\sqrt{dF(0,0)}, c^{**} \}.$$
In particular, in the slow predator case, statement~\textit{$(iii)$} of Theorem~\ref{THEO1} still holds true provided that $ c^{**} < 2\sqrt{dF(0,0)} $, even if $c^* > 2 \sqrt{dF(0,0)}$. On the other hand, in the fast predator case, we are not able to describe the behaviour of solutions in the range $2\sqrt{dF(0,0)}t \leq \| x \| \leq c^* t $; thus, we cannot exhibit in general a definite spreading speed without this additional KPP assumption.
\end{remark}

As mentioned in the outline of the paper, our argument is split into three steps. We first deal with pointwise weak spreading before dealing with pointwise spreading. We will then conclude the proof by establishing uniform spreading in the final zone.

\subsection{First step: Pointwise weak spreading}

The first step is to prove the following lemma, which states that the $u$-component of the solution converges to neither 0 and 1,  and the $v$-component does not converge to 0. Furthermore, this property is in some sense uniform with respect to any given bounded set $C\cap B_X (0,\kappa)$ of initial data.

\begin{lemma}\label{lem:1}
Let $\kappa >0$ be given. Then there exists $\varepsilon_1 (\kappa,c^0) >0$ such that, for any $(u_0,v_0) \in C\cap B_X (0,\kappa)$ with $u_0 \not \equiv 0$, for all $c \in [0,c^0]$, $e \in S^{N-1}$ and $x\in\mathbb{R}^N$:
$$
\displaystyle\limsup_{t \to +\infty} u (t,x+cte) \geq \varepsilon_1 (\kappa, c^0),
$$
and, if moreover $v_0 \not \equiv 0$,
$$\left\{
\begin{array}{l}
\displaystyle\limsup_{t \to +\infty} v (t,x+cte) \geq \varepsilon_1 (\kappa, c^0),\vspace{2pt}\\
\displaystyle\liminf_{t \to +\infty} u (t,x+cte) \leq 1 -\varepsilon_1 (\kappa,c^0).
\end{array}
\right.
$$
\end{lemma}
\begin{remark}
Note that it is immediate, from this statement, that $\varepsilon_1 (\kappa,c^0)$ can be chosen to be nonincreasing with respect to $\kappa$ and $c^0$. In particular, $\varepsilon_1 (\kappa, c^0) \geq \varepsilon_1 (M (\kappa),c^0)$ for any $\kappa >0$.
\end{remark}
\begin{proof}
Let us begin by noting that, when $v_0 \equiv 0$, then $u$ satisfies a standard monostable type equation and therefore it is well known that $u(t,x+cte)$ converges to 1 as $t \to +\infty$ for any $x \in \mathbb{R}^N$ and $c \in [0, c^*)$. Therefore we only need to consider the case when $v_0$ is not trivial.

We then argue by contradiction by assuming there exist sequences
\begin{equation*}
\begin{split}
&\left\{\left(u_{0,n},v_{0,n}\right)\right\}_{n\geq 0} \in (C \cap B_X (0,\kappa))^{\mathbb{N}},\\
&\{c_n\}_{n\geq 0}\subset  [0,c^0], \mbox{ and } \{e_n\}_{n\geq 0}\subset S^{N-1},\\
&\{x_n\}_{n\geq 0} \subset \mathbb{R}^N, \mbox{ and $\{t_n\}_{n\geq 0}\subset [0,\infty)$ such that } t_n \to +\infty,
\end{split}
\end{equation*}
such that $u_{0,n} , v_{0,n} \not \equiv 0$ and one of the three following options holds true:
\begin{equation}\label{eq:lem1:ineq120}
\forall t \geq t_n, \ u_n (t,x_n+c_n t e_n) \leq  \frac{1}{n},
\end{equation}
\begin{equation}\label{eq:lem1:ineq12}
\forall t \geq t_n, \ u_n (t,x_n+c_n t e_n) \geq 1 - \frac{1}{n},
\end{equation}
or
\begin{equation}\label{eq:lem1:ineq11}
\forall t \geq t_n, \ v_n (t,x_n+c_n t e_n) \leq \frac{1}{n},
\end{equation}
wherein $(u_n ,v_n)$ denotes the solution of \eqref{1.1} with initial data $\left(u_{0,n} , v_{0,n}\right)$. Note that without loss of generality, one may assume that
$$c_n \to c_\infty \in [0,c^0] \mbox{ and } e_n \to e_\infty \in S^{N-1}.$$

Let us first show that~\eqref{eq:lem1:ineq120} implies~\eqref{eq:lem1:ineq11}. Choose any sequence $s_n \geq t_n$. From the weak dissipativity assumption and standard parabolic estimates, possibly along a subsequence, the following convergence holds true
\begin{equation*}
\begin{cases}
\displaystyle \lim_{n\to \infty}v_n (s_n + t,x_n+c_n (s_n+t) e_n +x) \to v_\infty (t,x),\\
\displaystyle\lim_{n\to\infty}u_n (s_n +t,x_n+c_n (s_n+t) e_n + x) \to u_\infty (t,x).
\end{cases}
\end{equation*}
The above convergence is locally uniform in $(t,x)\in\R\times\R^N$ and the limit function $(u_\infty,v_\infty)$ is an entire solution of the following system, which is the same as \eqref{1.1} but with some additional drift term due to the fact we are looking at some moving frames:
\begin{equation}\label{1.1.drift}
\begin{split}
&\left(\partial_t -d\Delta\right)u_\infty= c_\infty \nabla u_\infty \cdot e_\infty+  u_\infty F\left(u_\infty , v_\infty\right),\\
&\left(\partial_t-\Delta \right)v_\infty = c_\infty  \nabla v_\infty \cdot e_\infty +v_\infty G\left(u_\infty, v_\infty\right).
\end{split}
\end{equation}
Next it is clear that $u_\infty \geq 0$ and we infer from \eqref{eq:lem1:ineq120} that $u_\infty (0,0) =0$. By the strong maximum principle, we get that $u_\infty \equiv 0$, hence $v_\infty$ satisfies
$$\partial_t v_\infty = \Delta v_\infty + c_\infty \nabla v_\infty \cdot e_\infty + v_\infty G\left(0,v_\infty \right).$$
It is clear, from the monotonicity of $G$, that for any $t_0\in\R$, the function $(t,x) \mapsto M (\kappa) e^{-G(0,0) (t+t_0)}$ is a super-solution of the same equation, for any $t > -t_0$. Since $v_\infty (-t_0 ,x) \leq M(\kappa)$ for any $t_0 \in \R_+$, it follows that
$$v_\infty (0,x) \leq M(\kappa) e^{-G(0,0) t_0}.$$
Passing to the limit as $t_0 \to +\infty$, we get that $v_\infty (0,x) \equiv 0$. Therefore, $v_n (s_n,x_n+c_n s_n e_n) \to 0$ as $n \to +\infty$ and, because the choice of the sequence $s_n$ was arbitrary, this means that~\eqref{eq:lem1:ineq11} holds.\\

Next we claim that 
\begin{claim}\label{Claim1}
In both cases, that is either \eqref{eq:lem1:ineq12} or \eqref{eq:lem1:ineq11} holds, then 
there exists a sequence $\{t_n'\}_{n\geq 0}$ such that $t_n'\geq t_n$ and
for any $R >0$,
\begin{equation}\label{eq:lem1:ineq22}
\lim_{n\to\infty} \sup_{t\geq 0,\;x\in B_R}|1-u_n (t_n' +t,x_n+ c_n (t_n' +t) e_n +x)|=0.
\end{equation}
Here $B_R$ denotes the closed ball of radius $R$ centred at $0$.
\end{claim}

The proof of this claim is split into two parts corresponding to the two cases \eqref{eq:lem1:ineq12} and \eqref{eq:lem1:ineq11}.\\

\begin{proof}[Proof of Claim~\ref{Claim1}]
Let us first consider the case when \eqref{eq:lem1:ineq12} holds true. 
We will show that Claim \ref{Claim1} holds true with the sequence $\{t_n'=t_n\}$. 
To do so we proceed by contradiction, and assume that for some $R>0$, there exist $\delta >0$, sequences $s_n \geq t_n$ and $x'_n \in B_R$ such that
\begin{equation}\label{eqplus}
u_n (s_n,x_n+c_n s_n e_n + x'_n) \leq 1- \delta.
\end{equation}
Using again the weak dissipativity assumption and standard parabolic estimates, we extract a converging subsequence
\begin{equation*}
\begin{cases}
\displaystyle \lim_{n\to \infty}v_n (s_n + t,x_n+c_n (s_n+t) e_n +x) \to v_\infty (t,x),\\
\displaystyle\lim_{n\to\infty}u_n (s_n +t,x_n+c_n (s_n+t) e_n + x) \to u_\infty (t,x),
\end{cases}
\end{equation*}
where the limit function $(u_\infty,v_\infty)$ is an entire solution of~\eqref{1.1.drift}. Next it is clear that $u_\infty \leq 1$ and we infer from \eqref{eq:lem1:ineq12} that $u_\infty (0,0) =1$. Thus $u_\infty \equiv 1$ (as well as $v_\infty \equiv 0$) by the strong maximum principle. But, since the sequence $\{x_n'\}\subset B_{R}$ is relatively compact, \eqref{eqplus} yields 
$$ u_\infty (0,x'_\infty) \leq 1 - \delta,$$
where $x'_\infty$ is an accumulation point of $\{x_n'\}$.
Hence we have reached a contradiction and Claim \ref{Claim1} holds true under condition \eqref{eq:lem1:ineq12}.\\

Let us now consider the case when \eqref{eq:lem1:ineq11} holds true. We begin by checking that for any $R>0$
\begin{equation}\label{eq:lem1:ineq21}
\lim_{n\to\infty} v_n (t_n +t,x_n+c_n (t_n+t) e_n +x)=0\text{ uniformly on } [0,\infty)\times B_R.
\end{equation}
The argument is the same as above. Indeed, if this is not true, then for some $R>0$, there exist $\delta >0$, $s_n \geq t_n$ and $x'_n \in B_R$ such that
$$v_n (s_n,x_n+c_n s_n e_n + x'_n) \geq \delta.$$
Once again due to weak dissipativity assumption and standard parabolic estimates
one may assume, possibly along a subsequence, that as $n\to\infty$:
\begin{equation*}
\begin{cases}
v_n (s_n + t,x_n+c_n (s_n+t) e_n +x) \to v_\infty (t,x),\\
u_n (s_n +t,x_n+c_n (s_n+t) e_n + x) \to u_\infty (t,x),
\end{cases}
\end{equation*}
where the above convergence holds locally uniformly and wherein $(u_\infty,v_\infty)$ is an entire solution of \eqref{1.1.drift}. Using the strong maximum principle and \eqref{eq:lem1:ineq11}, one can check that $v_\infty \equiv 0$, which contradicts the fact that $ v_\infty (0,x'_\infty) >\delta,$ where $x'_\infty$ is an accumulation point of the sequence $\{x'_n\}_{n\geq 0}$. We have proved~\eqref{eq:lem1:ineq21}.

We can proceed with the proof of~\eqref{eq:lem1:ineq22}. We have just shown that for any $R >0$ and $\delta >0$, then for any $n$ large enough one has for each $t\geq 0$ and $x\in\mathbb{R}^N$:
\begin{equation*}
v_n (t_n+t,x_n+c_n (t_n+t)e_n+x) \leq M (\kappa) \chi_{\mathbb{R}^N \setminus B_R } (x) + \delta \chi_{B_R} (x) =: \overline{v} (x).
\end{equation*}
Here, let us recall that $M(\kappa)$ is the weak dissipativity constant.

Next for each $R>0$, $\delta>0$ and $n$ large enough, we infer from the comparison principle and the monotonicity of $F$ with respect to $v$ that 
\begin{equation}\label{esti-inf}
u_n (t_n+t,x_n +c_n (t_n+t)e_n+x) \geq \underline{u}^n (t,x),\;\;\forall t\geq 0,\;x\in\mathbb{R^N},
\end{equation}
where $\underline{u}^n$ is the solution of
\begin{equation}\label{eq:underline_u}
\left\{
\begin{array}{l}
\left(\partial_t -d\Delta\right) \underline{u}^n= c_n \nabla \underline{u}^n \cdot e_n+  \underline{u}^n F\left( \underline{u}^n , \overline{v}\right),\vspace{3pt}\\
\underline{u}^n (0,x) := u_n (t_n,x_n+ c_n t_n e_n +x).
\end{array}
\right.
\end{equation}
For each $R>0$ let us denote by $\phi_R$ the principal (hence positive) eigenfunction of
$$
\left\{
\begin{array}{l}
\Delta \phi_R = \mu_R \phi_R \mbox{ in } B_R,\vspace{3pt}\\
\phi_R = 0 \mbox{ on } \partial B_R,
\end{array}
\right.
$$
normalized so that $\|\phi_R\|_\infty =1$, and extend it by $0$ outside of the ball $B_R$. Then, let us define the function $\psi_0(x;\eta)$, for each $\eta >0$, as 
\begin{equation*}
\psi_0(x; \eta) = \eta e^{-\frac{c_n}{2d} (x \cdot e_n )} \phi_R (x).
\end{equation*}
Recall also that for any $n\geq 0$, $0 \leq c_n \leq c^0 < \max \{c^*,c^{**} \}$ and thus, under assumptions of both Theorems~\ref{THEO1} and~\ref{THEO2} (see also Remark~\ref{rmk:KPP3}), $c^0 < 2 \sqrt{dF(0,0)}$.
Then using the same computations as in the previous section, one can check that there exists $\eta_0$ depending only on $F$ and $c^0$ such that, for any $\delta$ small enough, $0< \eta \leq \eta_0$ and $R$ large enough, $\psi_0$ satisfies for each $n\geq 0$:
\begin{equation*}
-d\Delta \psi_0 - c_n \nabla \psi_0 \cdot e_n -  \psi_0 F\left( \psi_0 , \delta \right) \leq 0\text{ in }B_R.
\end{equation*}
Moreover since ${\rm supp}\;\psi_0\subset B_R$, $\psi_0$ is also a sub-solution of equation \eqref{eq:underline_u} satisfied by $\underline{u}^n$.

Therefore, the solution, denoted by $\psi(t,x;\eta)$, of \eqref{eq:underline_u} associated with initial data $\psi(0,\cdot;\eta)=\psi_0 (\cdot;\eta)$ is increasing in time, and converges to some positive stationary solution that we denote by $p_{n,R,\delta}$. Let us check that it does not depend on the choice of $\eta \in (0, \eta_0]$. For this purpose, let us slightly change our notation for simplicity and denote this stationary solution by $p_\eta$, while $n$, $R$ and $\delta$ are fixed for the time being. First note that the comparison principle implies that $p_\eta \leq p_{\eta'}$ for any $\eta,\eta' \in (0,\eta_0]$ such that $\eta\leq \eta'$.
Next let us assume by contradiction that there exists $\eta_1<\eta_0$ with $p_{\eta_1}\not\equiv p_{\eta_0}$.
Hence the strong maximum principle implies that $p_{\eta_1}<p_{\eta_0}$. Moreover there exists a point $x_0\in B_R$ such that
\begin{equation*}
\psi(0,x_0;\eta_0)>p_{\eta_1}(x_0).
\end{equation*}
Indeed if not then $\psi(0,x;\eta_0)\leq p_{\eta_1}(x)$ for all $x\in\R^N$, which yields $p_{\eta_0}(x)\leq p_{\eta_1}(x)$, a contradiction.
Now consider $$\eta^*=\sup\{\eta\geq \eta_1:\;\psi(0,x;\eta)\leq p_{\eta_1}(x),\;\forall x\in\R^N\}.$$
Then one deduces from the comparison principle and the strong maximum principle that
\begin{equation*}
\psi(0,x;\eta^*)<\psi(t,x;\eta^*)<p_{\eta_1}(x),\;\forall t>0,\;x\in\R^N.
\end{equation*}
On the other hand, from the definition of $\eta^*$ and recalling that the functions $\psi$ have compact support $B_R$, there exists $x_0\in B_R$ such that $\psi(0,x_0;\eta^*)=p_{\eta_1}(x_0)$, a contradiction.
Hence it follows that $p_\eta \equiv p_{\eta_0}$, for all $\eta\in (0,\eta_0]$. As announced, it does not depend on $\eta$ and we now denote it by $p_{n,R,\delta}$.

Now, thanks to the fact that $u_n$ is not trivial, we can choose $\eta$ sufficiently small so that $\underline{u}^n (0,x) \geq \psi (0,x;\eta)$ for all $x\in\R^N$. Then it follows from \eqref{esti-inf} that for any $R>0$ large enough, $\delta>0$ small enough and $n$ large enough:
\begin{equation}\label{eq:liminf_pnrd}
\liminf_{t \to +\infty} u_n (t_n+t,x_n+c_n(t_n+t)e_n+x) \geq p_{n,R,\delta}(x),\;\;\forall x\in\mathbb{R}^N.
\end{equation}
To complete the proof of Claim \ref{Claim1},
it remains to check that $p_{n,R,\delta}$ is close enough to 1 as $n$ and $R$ are large and $\delta$ is small.

Since $p_{n,R,\delta}$ is also bounded from above by 1, one can use standard elliptic estimates to get that, as $n \to +\infty$, $R \to +\infty$ and $\delta \to 0$, the function $p_{n,R,\delta}$ converges (eventually up to extraction of a subsequence) locally uniformly to a stationary solution $p_{\infty}$ of 
\begin{equation}\label{eq:monostable_drift}
d \Delta p_{\infty} + c_\infty \nabla p_{\infty}  \cdot e_\infty + p_{\infty}  F (p_{\infty} ,0)=0.
\end{equation}
Moreover, recall that, since the map $t\mapsto \psi(t,x;\eta_0)$ is nondecreasing, $p_{n,R,\delta} (0) \geq \psi(0,0;\eta_0)\geq \eta_0 \phi_R (0)$. Now note that $\phi_R \to 1$ locally uniformly as $R \to +\infty$ (indeed $\phi_R (x) \equiv \phi_1\left(R^{-1} x \right)$), hence $p_{\infty} (0) \geq \eta_0$ and the strong maximum principle provides $p_\infty >0$. To conclude we shall make use of the following lemma for the monostable equation:
\begin{lemma}\label{lem:monostable_stationary}
Let $p=p(x)$ be a stationary of \eqref{eq:monostable_drift} such that $0<p(x)\leq 1$ for all $x\in\R^N$. Then $p(x)=1$, $\forall x\in\R^N$.
\end{lemma}
\begin{proof}[Proof of Lemma~\ref{lem:monostable_stationary}]
To check this, let $p\equiv p(x)$ be a stationary solution of \eqref{eq:monostable_drift} with $0<p(x)\leq 1$ for any $x\in\mathbb R^N$. Then note that the map $U(t,x):=p\left(x-c_\infty t e_\infty\right)$ satisfies 
\begin{equation*}
\partial_t U-d\Delta U=U F\left(U,0\right).
\end{equation*}
Next since $U(0,x)=p(x)>0$ and $c_\infty\in [0,c^*)$, one can make use of Proposition~\ref{LE1} to conclude that for each $x\in\mathbb R^N$ one has $U(t,x+c_\infty t e_\infty)\to 1$ as $t\to\infty$. Now note that $U(t,x+c_\infty t e_\infty)=p(x)$ for any $t\geq 0$ and $x\in\mathbb R^N$, that completes the proof of the lemma.
\end{proof}\\

The above lemma immediately implies that $p_\infty = 1$. Now, for any $R' >0$ and $\delta' >0$, one can choose $R$ and $\delta$ so that, for any $n$ large enough,
$$p_{n,R,\delta} (x) \geq 1 - \delta' \mbox{ for any } x \in B_{R' }.$$ 
Along with \eqref{eq:liminf_pnrd}, this completes the proof of Claim \ref{Claim1}.
\end{proof}\\

We can now go back to the proof of Lemma~\ref{lem:1}, by observing that \eqref{eq:lem1:ineq22} leads to a contradiction. The argument below is very similar to the one we just used to prove \eqref{eq:lem1:ineq22}, so that we omit some details.

First, to ease the notations, let us note that we can assume without loss of generality that $t'_n = t_n$ for all $n \in \mathbb{N}$. We now know that for any $R >0$ and $\delta >0$, then for any $n$ large enough, any $t\geq 0$ and $x\in\R^N$:
\begin{equation*}
u_n (t_n+t,x_n+c_n (t_n+t)e_n+x) \geq (1 -\delta) \chi_{B_R } (x) =: \underline{u} (x).
\end{equation*}
Then one infers from the comparison principle that 
\begin{equation*}
v_n (t_n+t,x_n +c_n (t_n+t)e_n+x) \geq \underline{v}^n (t,x),\;\forall t\geq 0,\;x\in\R^N,
\end{equation*}
wherein we have set $\underline{v}^n$, the solution of
\begin{equation}\label{eq:underline_v}
\left\{
\begin{array}{l}
\left(\partial_t - \Delta\right) \underline{v}^n= c_n \nabla \underline{v}^n \cdot e_n+  \underline{v}^n G\left( \underline{u} , \underline{v}^n\right),\vspace{3pt}\\
\underline{v}^n (0,x) = v_n (t_n,x_n+ c_n t_n e_n +x).
\end{array}
\right.
\end{equation}
Once again, the function $\psi_0 (x) = \eta e^{-\frac{c_n}{2} x \cdot e_n} \phi_R (x)$ is a sub-solution of the equation above, provided that $R$ is large and $\eta$ small, using the fact that $c^0 < c^{**}$. Thanks to the comparison principle, it follows, up to reducing $\eta$, that
\begin{equation}\label{eq:comp_vpsi}
v_n (t_n + t, x_n+c_n (t_n+t)e_n+x) \geq \underline{v}^n (t,x) \geq \psi (t,x),\;\forall t\geq 0,\;x\in\R^N,
\end{equation}
wherein $\psi=\psi(t,x)$ denotes the solution of \eqref{eq:underline_v} with initial data $\psi_0$. 
On the other hand, the function $\psi=\psi(t,x)$ again converges to some positive stationary solution, because it is increasing in time and bounded from above by $M(\kappa)$. As before, it can be proved that this positive stationary solution does not depend on small $\eta$, and we denote it by~$q_{n,R,\delta}$. 

Moreover, because it is bounded from above by $M(\kappa)$ and using standard elliptic estimates, $q_{n,R,\delta}$ converges as $n \to +\infty$, $R \to +\infty$ and $\delta \to 0$ (up to extraction of a subsequence and without loss of generality) to a stationary state $q_\infty$ of
$$\Delta q_{\infty} + c_\infty \nabla q_{\infty}  \cdot e_\infty + q_{\infty}  G (1,q_{\infty,})=0,$$
which is bounded and also positive thanks to the fact that $\phi_R \to 1$ locally uniformly as $R \to +\infty$. In fact, one even gets that
$$\inf_{x \in \mathbb{R}^N} q_\infty (x) >0,$$
with the same argument as for Lemma~\ref{lem:monostable_stationary}.

Note that for some $G$, namely such that $G(1,v) >0$ for any $v \geq 0$, this is already a contradiction as bounded positive stationary states do not even exist. However, such an equilibrium may exist in general under our assumptions, so that some more work is needed to reach the contradiction.

By inequality \eqref{eq:comp_vpsi}, 
$$\liminf_{t \to +\infty} v_n (t_n+t,x_n+c_n (t_n+t)e_n+x) \geq \frac{\inf q_\infty }{2}>0,$$
on the ball $B_R$, where $R >0$ can be chosen arbitrarily large, provided that $n$ is large enough. 
But due to the above lower estimate, one gets
$$\limsup_{t \to +\infty} u_n (t_n+t,x_n+c_n (t_n+t)e_n+x) \leq \lim_{t \to +\infty} \overline{u}_n (t,x),$$
wherein $\overline{u}_n$ is the solution of
$$
\left\{
\begin{array}{l}
\left(\partial_t - d\Delta\right) \overline{u}_n= c_n \nabla \overline{u} \cdot e_n+  \overline{u}_n F\left( \overline{u}_n , \frac{\inf q_\infty}{2} \chi_{B_R} (x) \right),\vspace{3pt}\\
\overline{u}_n (0,x) =1.
\end{array}
\right.
$$
Note that the limit of $\overline{u}_n$ as $t \to +\infty$ is well-defined since 1 is a super-solution, hence $\overline{u}_n$ decreases with respect to time. Furthermore, it can easily be checked with an argument similar to the above that this limit stays locally away from 1, uniformly with respect to $n$. Indeed, up to extraction of a subsequence, $\overline{u}_n(t,x)$ converges locally uniformly to the solution $\overline{u}_\infty(t,x)$ of the same problem where $c_n$ and $e_n$ are replaced respectively by $c_\infty$ and $e_\infty$. Because 1 still is a strict super-solution, one obtains that $\overline{u}_\infty (1,0) <1$.
Therefore
$$\limsup_{n \to \infty} \overline{u}_n (1,0) < 1,$$
and the monotonicity with respect to time yields 
$$\limsup_{n \to\infty} \lim_{t\to +\infty} \overline{u}_n (t,0) <1.$$
This contradicts Claim~\ref{Claim1} and the proof of the lemma is complete. \end{proof}

\subsection{Second step: Pointwise spreading}

We now prove the following improvement, that states that the solution spreads in any moving frame with constant speed uniformly with respect to bounded set of initial data:

\begin{lemma}\label{lem:2}
Let $\kappa >0$ be given. Then there exists $ \varepsilon_2 (\kappa, c^0) >0$ such that, for any $(u_0,v_0) \in C\cap B_X (0,\kappa)$ with $u_0 \not \equiv 0$ and $v_0 \not \equiv 0$, for all $c \in [0,c^0]$, $e \in S^{N-1}$ and any $x\in\mathbb R^N$:
$$\left\{
\begin{array}{l}
\displaystyle \liminf_{t \to +\infty} v (t,x+cte) \geq \varepsilon_2 (\kappa, c^0),\vspace{2pt}\\
\displaystyle \liminf_{t \to +\infty} u (t,x+cte) \geq \varepsilon_2 (\kappa, c^0),\vspace{2pt}\\
\displaystyle \limsup_{t \to +\infty} u (t,x+cte) \leq 1 -\varepsilon_2 (\kappa ,c^0) .
\end{array}
\right.
$$
\end{lemma}

Of course, the same remark as for Lemma~\ref{lem:1} holds here. Note also that the proof of this lemma uses ideas similar to what can be found in uniform persistence theory in dynamical systems.
We refer the reader for instance to Proposition 3.2 derived by Magal and
Zhao \cite{Magal-Zhao} and to the monograph of Smith and Thieme \cite{Smith-Thieme}.\\

\begin{proof}
We again proceed by contradiction to prove the first assertion, namely that $v$ spreads away from 0. We assume that there exist sequences $(u_{0,n},v_{0,n}) \in C \cap B_X (0,\kappa)$ (both components being nontrivial), $c_n \in [0,c^0]$, $e_n \in S^{N-1}$ and finally, $x_n \in \mathbb{R}^N$, such that
$$
\liminf_{t \to +\infty} v_n (t,x_n+c_nte_n) < \frac{1}{n}.
$$
We assume, without loss of generality, that
$$c_\infty = \lim_{n \to +\infty} c_n \in [0, c^0] \mbox{ and } e_\infty = \lim_{n \to +\infty} e_n \in S^{N-1}.$$
By the previous Lemma~\ref{lem:1}, there exist two sequences $t_n \to +\infty$ and $s_n \in \mathbb{R}_+$ such that for each $n\geq 0$
$$v_n (t_n,x_n + c_nt_n e_n) = \frac{\varepsilon}{2},$$
$$v_n (t,x_n + c_nt e_n) \leq \frac{\varepsilon}{2}, \   \ \forall t \in [t_n,t_n+s_n],$$
$$v_n (t_n + s_n,x_n + c_n( t_n +s_n) e_n) = \frac{1}{n},$$
where $\varepsilon = \varepsilon_1 (M(\kappa), c^0)$ is provided by Lemma \ref{lem:1}.

As before, possibly along a subsequence, the functions $u_n (t_n+s_n+t, c_n (t_n+s_n)e_n +x)$ and $v_n (t_n+s_n+t,x_n+c_n (t_n+s_n)e_n +x)$ converge locally uniformly to $(u_\infty,v_\infty)$, an entire solution of~\eqref{1.1}. From the choice of $t_n$ and $s_n$, one has $v_\infty (0,0) = 0$, hence $v_\infty \equiv 0$ by the strong maximum principle. In particular, the sequence $s_n$ may not be bounded as it would contradict the fact that
$$\lim_{n \to +\infty} v_n (t_n,x_n+c_n t_n e_n) = \frac{\varepsilon}{2} >0.$$
We can thus assume that $s_n \to +\infty$ as $n \to +\infty$.\\

Now let us consider the limit functions
$$\tilde{u} (t,x) = \lim_{n \to +\infty} u_n (t_n  +t,x_n+c_n t_n e_n +x),$$
$$\tilde{v} (t,x) = \lim_{n \to +\infty} v_n (t_n  +t,x_n+c_n t_n e_n +x),$$
which are well defined thanks to weak dissipativity and parabolic estimates (as always, up to extraction of another subsequence). The pair $(\tilde{u},\tilde{v})$ is a global in time solution of system~\eqref{1.1}, and moreover $\tilde{v} (0,0) = \frac{\varepsilon}{2} >0$. 

Proceeding as in the beginning of the proof of Lemma~\ref{lem:1}, one can also check that $\tilde{u} (t=0) \not \equiv 0$. Indeed, assume by contradiction that $\tilde{u} (t=0) \equiv 0$, thus $\tilde{u} \equiv 0$ by the strong maximum principle. Then $\tilde{v}$ is a bounded solution of $(\partial_t - \Delta) \tilde{v} = \tilde{v} G(0,\tilde{v})$, which may only be 0 because $G(0,v) \leq G(0,0)<0$ for all $v \geq 0$. This would contradict the fact that, by construction, $\tilde{v} (0,0)>0$.

Then we look at $(\tilde{u},\tilde{v})$ as a solution of system~\eqref{1.1} with initial data
$$(\tilde{u}_0 ,\tilde{v}_0) := \lim _{n \to +\infty} (u_n (t_n  ,x_n+c_n t_n e_n +x),v_n (t_n ,x_n+c_n t_n e_n +x)),$$
which belongs to $C \cap B_X(0,M(\kappa))$, and whose both components are nontrivial. Applying again Lemma~\ref{lem:1}, one gets that
$$\forall x \in \mathbb{R}^N, \ \limsup_{t\to +\infty} \tilde{v} (t,x+cte) \geq \varepsilon,$$
for any $c \in [0,c^0]$ and $e\in S^{N-1}$.

On the other hand, for all $t \in [0,s_n)$,
$$v_n (t_n+t, x_n +c_n t_n e_n + c_n te_n) \leq \frac{\varepsilon}{2}.$$
Since $s_n \to +\infty$, we get by the locally uniform convergence that
$$\tilde{v} (t,c_\infty te_\infty) \leq \frac{\varepsilon}{2},\;\forall t\geq 0,$$
which contradicts the inequality above provided by Lemma~\ref{lem:1}.\\

The second and third assertions, that is the spreading properties of $u$, can now be proved with similar arguments.
\end{proof}

\subsection{Third step: Uniform spreading}

In the previous subsection, we have shown that some pointwise spreading property occurs locally in any frame moving with a constant speed $0 \leq c < c^*$. We now prove that this spreading is in fact uniform on the whole balls of radius $ct$ and center at the origin, with $0 \leq c < \min \{c^*, c^{**} \}$.

\begin{lemma}\label{lem:3}
Let an initial data $(u_0,v_0)\in C \cap B_X (0,\kappa)$ with $u_0 \not \equiv 0$ and $v_0 \not \equiv 0$ be given for some $\kappa>0$. Then for any $0 \leq c^0 < \min \{c^*, c^{**}\}$, there exists $\varepsilon>0$ such that
$$\liminf_{t\to +\infty} \inf_{\| x\| \leq c^0 t} v(t,x) \geq \varepsilon,$$
$$\liminf_{t\to +\infty} \inf_{\| x\| \leq c^0 t} u(t,x) \geq \varepsilon,$$
$$\limsup_{t\to +\infty} \sup_{\| x\| \leq c^0 t} u(t,x) \leq 1-\varepsilon.$$
\end{lemma}

\begin{proof}
Once more, we proceed by contradiction. Assume that there exist $t_n \to +\infty$, $c_n \in [0,c^0]$ and $e_n \in S^{N-1}$ such that, for instance
\begin{equation}\label{eq:lem31}
v(t_n, c_n t_n e_n) \to 0.
\end{equation}
Without loss of generality, up to a subsequence, we assume that $c_n \to c_\infty\in [0,c^0]$ and $e_n \to e_\infty$ as $n \to +\infty$. Choose some small $\delta >0$ such that $c_\infty + \delta < \min \{ c^* , c^{**} \}$, and define the sequence
$$t'_n := \frac{c_n t_n}{c_\infty + \delta} \in [0,t_n),\;\forall n\geq 0.$$

Consider first the case when the sequence $\{c_n t_n\}_{n\geq 0}$ is bounded, which may happen if $c_\infty =0$. Then up to extraction of a subsequence, one has as $n \to +\infty$ that 
$$c_n t_n e_n \to x_\infty \in \R^N,$$
and
$$v (t_n +t, c_n t_n e_n + x) \to 0,$$
locally uniformly. This directly follows from the strong maximum principle and the fact that any limit $\tilde{v} (t,x)$ satisfy $\tilde{v} (0,x_\infty) =0$. Thus, one also obtains that $v(t_n,0) \to 0$ as $n \to +\infty$, which already contradicts Lemma~\ref{lem:2} with $c=0$.

We can now assume that $t'_n \to +\infty$. Then due to Lemma~\ref{lem:2} one obtains
\begin{equation}\label{eq:lem32}
v (t'_n,(c_\infty+\delta) t'_n e_\infty) \geq \varepsilon,
\end{equation}
for each $n$ large enough, where $\varepsilon = \varepsilon_2 (M(\kappa),c_\infty + \delta)$ is the constant provided by Lemma \ref{lem:2}.

Now let us look at the functions
$$\tilde{v}_n (t,x) = v (t'_n + t, c_n t_n e_\infty + x)\text{ and }\tilde{u}_n (t,x) = u (t'_n + t, c_n t_n e_\infty  + x).$$
Define also the sequences
$$\tilde{c}_n := \frac{c_n t_n \| e_n - e_\infty \|}{t_n - t'_n} \to 0 \; \text{ and }\tilde{e}_n := \frac{e_n - e_\infty}{\|e_n - e_\infty\|}.$$
Using the above notations, \eqref{eq:lem31} and \eqref{eq:lem32} rewrite as
$$\tilde{v}_n (0,0) \geq \varepsilon \ \mbox{ and } \tilde{v}_n (t_n - t'_n,\tilde{c}_n (t_n - t'_n) \tilde{e}_n) \to 0.$$
Now let us introduce the sequences 
$$\tilde{t}_n := \sup \left\{ 0 \leq t \leq t_n - t'_n \; | \ \tilde{v}_n (t,\tilde{c}_n t \tilde{e}_n) \geq \frac{\varepsilon}{2} \right\} \in (0, t_n - t'_n),$$
$$\tilde{s}_n := t_n - t'_n - \tilde{t}_n.$$
Then this yields the following properties:
$$\tilde{v}_n (\tilde{t}_n , \tilde{c}_n \tilde{t}_n \tilde{e}_n) = \frac{\varepsilon}{2},$$
$$\tilde{v}_n (t,\tilde{c}_n t \tilde{e}_n) \leq \frac{\varepsilon}{2}, \ \ \forall t \in [\tilde{t}_n, \tilde{t}_n + \tilde{s}_n],$$
$$\tilde{v}_n (\tilde{t}_n+\tilde{s}_n, \tilde{c}_n (\tilde{t}_n + \tilde{s}_n) \tilde{e}_n) \to 0 \ \mbox{ as } n \to +\infty.$$
Proceeding as in the proof of Lemma~\ref{lem:2}, one reaches a contradiction with Lemma~\ref{lem:1}.  Thus we conclude that
$$\liminf_{t \to +\infty} \inf_{\| x \| \leq c^0 t } v(t,x) \geq \varepsilon .$$
The proof of the second and third statements in Lemma \ref{lem:3}, namely $\displaystyle \liminf_{t\to +\infty} \inf_{\| x \| \leq c^0 t} u (t,x) \geq  \varepsilon$ and $\displaystyle \limsup_{t\to +\infty} \sup_{\| x \| \leq c^0 t} u (t,x) \leq 1 - \varepsilon$, follow from the same arguments. The proof of Lemma~\ref{lem:3} is complete.
\end{proof}
\bigskip

\section{Weak dissipativity}\label{sec:dissip}

The conclusions of our main theorems hold provided that the considered system~\eqref{1.1} satisfies the weak dissipativity property (see Assumption \ref{ASS-U}). We believe this assumption to be satisfied by large classes of systems, and the following theorem gives a sufficient condition for weak dissipativity.
\begin{theorem}\label{th:dissip}
Let Assumptions \ref{ASS-F} and \ref{ASS-G} be satisfied.
Set $$m^* := \inf \{ 0 \leq m \leq 1 \; | \ \forall u \geq m, \; F (u,+\infty) < 0 \}.$$
If $G(0,+\infty)>-\infty$ and $G(m^*,+\infty) < 0$, then the weak dissipativity Assumption~\ref{ASS-U} is satisfied. Conversely, if it is satisfied, then $G(m^*,+\infty )\leq 0$.
\end{theorem}
Note first that $F(\cdot,+\infty)$ is well-defined in $\{-\infty\}\;  \cup \; \mathbb{R}$ thanks to the monotonicity of $F$, and so is $m^*$ since $F(1,+\infty) < 0$. The constant $m^*$ can be described as the largest stable state of the prey's population when the predator's is infinitely dense. 
%

Furthermore, the following corollary immediately follows from Theorem~\ref{th:dissip}.
\begin{corollary}\label{cor:dissip}
If $F(u,+\infty)< 0$ for all $u \in (0,1]$ and $G(0,+\infty)> -\infty$, then the weak dissipativity Assumption~\ref{ASS-U} holds true. 
\end{corollary}
Indeed, if $F(\cdot, +\infty) < 0$ in $(0,1]$, then $m^* = 0$ and by our assumptions $G(0,+\infty) \leq G(0,0)<0$.

We point out that the prey-predator models mentioned in the introduction (see~\eqref{example} and \eqref{Lotka}) clearly satisfy the conditions of Corollary~\ref{cor:dissip}. However, we expect weak dissipativity to hold for many other systems that are not covered by the above theorem and corollary, for instance if $G(0,v) \equiv 0$. This case arises in combustion in adiabatic environments.
\begin{remark}
As we mentioned in the introduction, the weak dissipativity is easily verified if we assume that $G(1,+\infty)<0$, which is satisfied by models which involve additional intraspecific competition in the predator. However, typical predator models do not assume such intraspecific competition, and we usually have $G(1,+\infty)>0$, which makes the verification of weak dissipativity much harder.
\end{remark}

\begin{proof}[Proof of Theorem \ref{th:dissip}]
Let us first assume that $G(m^*,+\infty) < 0$ and prove the weak dissipativity. 
Note that the condition $G(0,+\infty)>-\infty$ ensures that $G$ is bounded on $[0,1]\times [0,\infty)$.

We fix some $\kappa >0$ and recall that, from the comparison principle, for any initial data $(u_0,v_0) \in C\cap B_X (0,\kappa)$, the associated solution $(u,v)$ satisfy
$$0 \leq u \leq 1 \mbox{ and } v \geq 0.$$
Thus, we only need to prove that $v$ is uniformly bounded by some $M(\kappa)$ which depends on $\kappa$ but not on the specific choice of the initial data~$(u_0,v_0)$. To argue by contradiction, suppose that there exists $(u_{0,n} ,v_{0,n} ) \in C \cap B_X (0,\kappa)$ a sequence of initial data such that
$$\exists t_n \geq 0, x_n \in \R^N \mbox{ such that } v_n (t_n,x_n) \to +\infty \mbox{ as } n \to +\infty,$$
where for each $n\geq 1$, $(u_n,v_n)$ denotes the solution of the Cauchy problem \eqref{1.1} with initial data $(u_{0,n} ,v_{0,n} )$.

One can easily check that the function $\overline{v} := \kappa e^{G(1,0)t}$ is a super-solution of
$$\partial_t \overline{v} \geq  \Delta  \overline{v} +  \overline{v} G(u_n,\overline{v}),$$
for any $n \in \mathbb{N}^*$. By the comparison principle, $\| v_n (t,\cdot) \|_{L^\infty (\R^N)}$ is well-defined for any $t >0$, and is locally uniformly bounded with respect to $t$.

In particular, we obtain that $t_n \to +\infty$ as $n \to +\infty$, and we can assume, without loss of generality, that
$$t_n = \min \{ t > 0 \; | \ \| v_n (t,\cdot) \|_{L^\infty (\R^N )} =n \}\text{ and }v_n (t_n,x_n) \in \left[ \frac{n}{2} ,n \right].$$
Next let us denote
$$\tilde{v}_n (t,x) := \frac{v_n (t_n+t,x_n+x)}{v_n (t_n,x_n)},$$
which is a solution of
$$\partial_t \tilde{v}_n (t,x) = \Delta \tilde{v}_n (t,x) + \tilde{v}_n (t,x) G (u_n (t_n+t,x_n+x),v_n (t_n+t,x_n+x)).$$
This sequence of functions is also locally uniformly bounded: this is clear for negative $t$ by construction, and for positive $t$ it follows from the same super-solution $\overline{v}$ as above. Then, by parabolic estimates, one can extract a converging subsequence to some $\tilde{v}_\infty$ an entire (weak) solution of
\begin{equation}\label{eq:dissip_vlim_eq}\partial_t \tilde{v}_\infty (t,x) = \Delta  \tilde{v}_\infty (t,x) + \tilde{v}_\infty (t,x) G_\infty (t,x),
\end{equation}
where 
$$G_\infty (t,x) := \lim_{n \to +\infty} G (u_n (t_n+t,x_n+x),v_n (t_n+t,x_n+x)),$$
this limit being well-defined in the $L_{loc}^\infty\left(\mathbb R\times \mathbb R^N\right)-$weak star topology, thanks to the boundedness of $G$.

By construction, $\tilde{v}_\infty (0,0)=1$ and thus, by the strong maximum principle, it is positive everywhere. We conclude that \begin{equation}\label{eq:dissip_vlim}
v_n (t_n+t,x_n+x) \to +\infty,
\end{equation}
locally uniformly as $n \to +\infty$.

This will now allow us to obtain a contradiction, in a fashion similar to our arguments in the previous section. Indeed, since $u_n (t_n+t, x_n +x)$ is uniformly bounded from above by 1, we can define
$$u_\infty (t,x) := \limsup_{n \to +\infty} u_n (t_n+ t, x_n+x) \in [0,1].$$
Let us prove that $u_\infty \leq m^*$. To do so we fix any $\delta >0$ and we shall check that $u_\infty \leq m^* + \delta$.

To proceed let us consider $M$ so that $\displaystyle\sup_{u\in [m^* +\delta,1]} F (u,M) < 0$. Next, thanks to \eqref{eq:dissip_vlim}, we know that for any $T >0$ and $R >0$ and for any $n$ large enough,
$$u_n (t_n+ t , x_n + x) \leq \overline{u}_{T,R} (t,x) \ \mbox{ for }  \| x\| \leq R \mbox{ and } |t | < T,$$
wherein $\overline{u}_{T,R}$ satisfies
\begin{equation*}
\left\{
\begin{array}{l}
\partial_t \overline{u}_{T,R} (t,x) = d \Delta \overline{u}_{T,R} (t,x) + \overline{u}_{T,R} F (\overline{u}_{T,R} ,M), \ \mbox{ for }  \| x\| \leq R \mbox{ and } |t | < T, \vspace{3pt} \\
\overline{u}_{T,R} (-T,x) = 1, \ \mbox{ for } \| x \| \leq R, \vspace{3pt}\\
\overline{u}_{T,R} (t,x) = 1, \ \mbox{ for } \| x \| = R \mbox{ and } |t| < T. \vspace{3pt} \\
\end{array}
\right.
\end{equation*}
It is clear, from our choice of $M$, that 
$$\limsup_{T \to +\infty, R \to +\infty} \overline{u}_{T,R} (t,x) \leq m^* + \delta,\;\forall t\in\R,\;x\in\R^N. $$
Thus as announced we obtain 
\begin{equation}\label{eq:dissip_ulim}
u_\infty \leq m^*.
\end{equation}
Now, using \eqref{eq:dissip_vlim} and \eqref{eq:dissip_ulim}, we know that
$$G_\infty (t,x) \leq G(m^*,+\infty).$$
Therefore the function 
$$\overline{v}(t,x) := 2 e^{G(m^*,+\infty) (t+T)}$$
is a super-solution for \eqref{eq:dissip_vlim_eq} satisfied by $\tilde{v}_\infty$, for any $T \in \mathbb{R}$. From our choice of~$t_n$, it is clear that $\tilde{v}_\infty (-T,\cdot) \leq 2 \equiv \overline{v} (-T,\cdot)$, provided that $T \geq 0$. Thus the comparison principle yields 
$$\tilde{v}_\infty (0, 0) \leq 2e^{G(m^*,+\infty)T}.$$
Recalling that $G(m^*,+\infty)<0$, letting $T\to\infty$ contradicts $\tilde{v}_\infty (0,0)=1$.
This completes the proof of the first part of the theorem.\\

To prove the second part of the theorem, let us assume that $G(m^*,+\infty) >0$ and consider the initial data $(u_0 ,v_0) \equiv (m^*, 1)$. We can then ignore the space dependence and diffusion, as the solution satisfies the system of ODEs:
\begin{equation*}
\begin{cases}
\partial_t u (t)=u(t)F\left(u(t),v(t)\right),\\
\partial_t v(t)=v(t)G\left(u(t),v(t)\right).
\end{cases}
\end{equation*}
One can easily check, from the definition of $m^*$, that $m^* F\left(m^*,v\right) \geq 0$ for any $v \geq 0$, hence $u (t) \geq m^*$ for any $t \geq 0$. Therefore, $\inf_{t >0} G(u (t),v(t)) \geq G(m^*,+\infty) >0$, and it immediately follows that $v(t) \to +\infty$ as $t \to +\infty$. Thus the weak dissipativity Assumption~\ref{ASS-U} does not hold, and Theorem~\ref{th:dissip} is proved.
\end{proof}

%
%
%

\section{Asymptotic behaviour in the final zone}\label{sec:asymp}

We complete this paper with a short proof of Theorem~\ref{THEO.asymptotics} as well as two examples to which our result applies. In particular, we now assume that the diffusivity ratio $d=1$, and that the ODE system~\eqref{1.1.ODE} admits a unique stable and positive equilibrium point $(u^* ,v^*)$ as well as a Lyapunov function $\Phi$, whose properties are stated in Assumption~\ref{ass:Lyapunov}.\\

Since $\Phi$ is bounded from below, we assume without loss of generality that $\Phi\geq 0$, and it is an equality only at the unique minimizer $(u^* ,v^*)$.

In order to prove Theorem~\ref{THEO.asymptotics}, let us argue by contradiction by assuming that there exist $c\in \left[0,\min\{c^*,c^{**}\}\right)$, a sequence $\{\left(t_k,x_k\right)\}_{k\geq 0}\subset (0,\infty)\times\R^N$ such that $t_k \to +\infty$ and $\delta>0$ such that for all $k\geq 0$:
\begin{equation}\label{popol}
\|x_k\|\leq ct_k\text{ and }|u(t_k,x_k)-u^*|+ |v(t_k,x_k)-v^*|\geq \delta.
\end{equation}
Consider the sequence of functions $(u_k,v_k)$ defined by
\begin{equation*}
(u_k,v_k)(t,x)=(u,v)(t+t_k, x+x_k).
\end{equation*}
Let us fix $c'>0$ such that $c<c'<\min\{c^*,c^{**}\}$.
Note that the uniform inner spreading speed stated in Theorems \ref{THEO1} and \ref{THEO2} apply and there exist $M>0$, $A>0$ large enough and $\varepsilon>0$ small enough such that
for $k\geq 0$, $t\in\R$ and $x\in\R^N$ one has
\begin{equation}\label{pol}
t+t_k\geq A\text{ and }\|x\|\leq c't+(c'-c)t_k\;\Rightarrow\;\begin{cases} \varepsilon\leq u_k(t,x)\leq 1-\varepsilon\\ \varepsilon\leq v_k(t,x)\leq M.\end{cases}
\end{equation}
Now by parabolic estimates, possibly along a subsequence, one may assume that
\begin{equation*}
(u_k,v_k)(t,x)\to \left(u_\infty,v_\infty\right)(t,x)\text{ locally uniformly for $(t,x)\in\R\times\R^N$},
\end{equation*} 
where $\left(u_\infty,v_\infty\right)$ is a bounded entire solution of \eqref{1.1}. 
In addition, due to \eqref{pol}, the function $(u_\infty,v_\infty)$ satisfies
\begin{equation}\label{cond-inf}
\begin{split}
&\inf_{(t,x)\in\R\times \R^N}u_\infty(t,x)>0 \ \ \text{ and }\inf_{(t,x)\in\R\times \R^N}v_\infty(t,x)>0,\\
&\sup_{(t,x)\in\R\times \R^N}u_\infty(t,x)<1,
\end{split}
\end{equation}
while \eqref{popol} ensures that
\begin{equation}\label{popopol}
|u_\infty(0,0)-u^*|+ |v_\infty(0,0)-v^*|>0.
\end{equation}
In order to reach a contradiction, we claim that:
\begin{claim}\label{claim-popol}
Let $(U,V)$ be a bounded entire solution of \eqref{1.1} satisfying \eqref{cond-inf}. Then
\begin{equation*}
(U,V)(t,x)\equiv (u^*,v^*).
\end{equation*} 
\end{claim}

Note that this claim contradicts \eqref{popopol} and thus completes the proof of Theorem~\ref{THEO.asymptotics}.
It remains to prove Claim \ref{claim-popol}.

To do so let us consider
\begin{equation*}
W (t,x):=\Phi\left(U(t,x),V(t,x)\right).
\end{equation*}
Then one has
\begin{equation}\label{sub-sol}
\begin{split}
\partial_t W-\Delta W =&- \left( \Phi_{uu} | \nabla U |^2 + 2 \Phi_{uv} \nabla U \cdot \nabla V + \Phi_{vv} |\nabla V |^2 \right) \vspace{3pt}\\
&+ \Phi_u UF(U,V) + \Phi_v V G(U,V) \\
  \leq & \hspace{1mm} 0.
\end{split}
\end{equation}
Because of \eqref{cond-inf}, the function $W$ is uniformly bounded.
Consider a sequence $\left\{\left(t_n,x_n\right)\right\}_{n\geq 0}\subset \R\times \R^N$ such that
\begin{equation}\label{cond-sup}
\lim_{n\to\infty} W(t_n,x_n)=\sup_{(t,x)\in\R\times\R^N} W(t,x).
\end{equation}
Next consider the sequence $W_n(t,x)=W\left[U,V\right]\left(t+t_n,x_n+x\right)=W\left[U_n,V_n\right](t,x)$ wherein we have set for each $n\geq 0$:
\begin{equation*}
\left(U_n,V_n\right)(t,x)=\left(U,V\right)\left(t_n+t,x_n+x\right),\;\;\forall (t,x)\in\R\times \R^N.
\end{equation*}
Due to parabolic estimates, without loss of generality, one may assume, possibly along a subsequence, that $\left(U_n,V_n\right)\to \left(U_\infty,V_\infty\right)$ locally uniformly and $W_n\to W_\infty:=W\left[U_\infty,V_\infty\right]$ locally uniformly where $\left(U_\infty,V_\infty\right)$ is an entire solution of~\eqref{1.1} satisfying \eqref{cond-inf}.
Note that $W_\infty$ satisfies:
\begin{equation*}
W_\infty(0,0)=\sup_{(t,x)\in\R\times\R^N} W(t,x)=\sup_{(t,x)\in\R\times\R^N} W_\infty(t,x).
\end{equation*}
Moreover \eqref{sub-sol} ensures that $W_\infty$ is a sub-solution of the heat equation. Therefore the maximum principle applies and $W_\infty$ is a constant function. Hence the right hand side of \eqref{sub-sol} together with the strict convexity of $\Phi$ provide: 
\begin{equation*}
\begin{cases}
U_\infty(t,x)\equiv U_\infty(t),\;\;V_\infty(t,x)\equiv V_\infty(t),\\
\left(U_\infty F\left(U_\infty,V_\infty\right),V_\infty G\left(U_\infty,V_\infty\right)\right) \cdot \nabla \Phi \left(U_\infty,V_\infty\right)\equiv 0.
\end{cases}.
\end{equation*}
In addition, $\Phi\left(U_\infty(t),V_\infty(t)\right)=\Phi\left(U_\infty(0),V_\infty(0)\right)$ for all $t\in\R$. Since the Lyapunov function is assumed to be strict (see Assumption \ref{ass:Lyapunov} \textit{$(b)$}), we obtain that $U_\infty(t)\equiv u^*$ and $V_\infty(t)\equiv v^*$.
Hence one obtains that
\begin{equation*}
0\leq W[U,V](t,x)\leq \Phi\left(u^*,v^*\right)=0.
\end{equation*}
This proves Claim \ref{claim-popol}, completing the proof of Theorem~\ref{THEO.asymptotics}.\\

We complete this section by giving two examples to which the above theorem applies. We consider the simple Lotka-Volterra prey-predator system with logistic growth of the prey population proposed in \eqref{Lotka} and 
the so-called Holling type II prey-predator system.

\begin{example}[Lotka-Volterra prey-predator system]
The Lotka-Volterra prey-predator system is presented in \eqref{Lotka}. 
With this notation the kinetic system is given by the ODE system \eqref{1.1.ODE} with
\begin{equation}\label{vector-field}
F(u,v)=1-u-b v\;\;\text{ and }\;\;G(u,v)=\mu bu-a,
\end{equation}
where $a>0$, $b>0$ and $\mu>0$ are given parameters.
We assume that $\mu b>a$ and we set $\left(u^*,v^*\right)=\left(\frac{a}{\mu b},\frac{\mu b-a}{\mu b^2}\right)\in \mathcal O$, the unique positive stationary state.
Note that $F(\cdot,+\infty) <0$ and $G(0,0)=-a<0<\mu b-a=G(1,0)$.
Hence function $\left(F,G\right)$ satisfies Assumptions \ref{ASS-F}, \ref{ASS-G} as well as \ref{ASS-U} thanks to Theorem~\ref{th:dissip} in Section~\ref{sec:dissip}.
We now aim at applying Theorem \ref{THEO.asymptotics}. To do so, let us consider the strictly convex functional 
\begin{equation*}
\Phi(u,v):=\mu\int_{u^*}^u\frac{\xi-u^*}{\xi}d\xi+\int_{v^*}^v \frac{\eta-v^*}{\eta}d\eta.
\end{equation*}
Then it is easy to check that
\begin{equation*}
\left(uF(u,v),vG(u,v)\right) \cdot \nabla \Phi (u,v)=-\mu\left(u-u^*\right)^2 \leq 0 ,\;\forall (u,v)\in\mathcal O.
\end{equation*}
Furthermore, it easily follows from the above expression that $\Phi$ is a strict Lyapunov function in the sense of Assumption \ref{ass:Lyapunov} \textit{$(b)$}.
Hence we conclude that Theorem~\ref{THEO.asymptotics} applies, so that for any compactly supported initial data, the solution of the PDE system with diffusivity ratio $d=1$ converges inside the final zone to the constant equilibrium $(u^*, v^*)$.
\end{example}

\begin{example}[Holling II prey-predator system]
Holling II prey-predator system corresponds to problem \eqref{predator-prey} where
\begin{equation*}
h(u)=1-u \text{ and } \Pi(u)=\frac{mu}{b+u}.
\end{equation*}
We refer to~\cite{Holling} for a background on this model. If $\frac{\mu m}{b+1}>a$ then $G(0,0)=-a<0<\frac{\mu m}{b+1}-a$ and the ODE system has a unique positive stationary state $\left(u^*,v^*\right)\in \mathcal O$ with 
$$
v^*=\frac{1}{m}(1-u^*)(b+u^*)\text{ and }\mu\Pi(u^*)=a. 
$$
Moreover, $F(\cdot,+\infty)<0$. Hence, as in the previous example, Assumptions \ref{ASS-F}, \ref{ASS-G} and \ref{ASS-U} are satisfied. Now following \cite{CHL}, we consider the strictly convex function $\Phi$ defined on $\mathcal O$ by
\begin{equation*}
\Phi(u,v)=\int_{u^*}^u \mu \frac{\Pi(\xi)-\Pi(u^*)}{\Pi(\xi)}d\xi+\int_{v^*}^v \frac{\eta-v^*}{\eta}d\eta,
\end{equation*}
so that one has
\begin{equation*}
\begin{split}
J(u,v):&=\left(uF(u,v),vG(u,v)\right) \cdot \nabla \Phi (u,v)\\
&=\frac{\mu}{m}\left(\Pi(u)-\Pi(u^*)\right)\left((1-u)(b+u)-mv^*\right).
\end{split}
\end{equation*}
Hence one gets $J(u,v)\leq 0$ for all $(u,v)\in\mathcal O$ as soon as $b\geq mv^*$.
Finally, as in the previous example, we conclude that Theorem~\ref{THEO.asymptotics} applies when the following conditions are satisfied
\begin{equation*}
\frac{\mu m}{b+1}>a\text{ and }b\geq mv^*.
\end{equation*}
And thus, in that case, for compactly supported initial data, the solution of the Holling II prey-predator system with diffusivity ratio $d=1$ converges  to $(u^*,v^*)$ inside the final zone.

\end{example}

\end{document}